\documentclass[11pt]{article}
\usepackage{amssymb,enumerate}
\usepackage{amsmath,amsfonts}
\usepackage{amsthm}
\usepackage{latexsym}
\usepackage{mathrsfs}
\usepackage{color}
\usepackage{microtype}
\usepackage{hyperref}
\usepackage{natbib}
\usepackage{graphicx}

\usepackage{geometry}
\usepackage{etoolbox}

\usepackage[linesnumbered,ruled,vlined]{algorithm2e}
\usepackage{xcolor}
\usepackage{multirow}

\allowdisplaybreaks

\DeclareMathOperator{\cO}{\ensuremath{\mathcal{O}}}

\DeclareMathOperator{\bR}{\ensuremath{\mathbb{R}}}

\DeclareMathOperator{\sgn}{\ensuremath{\mathrm{sgn}}}

\newtheorem{lemma}{Lemma}

\newtheorem{assumption}{Assumption}

\newtheorem{theorem}{Theorem}

\newtheorem{proposition}{Proposition}

\newcommand{\beq}{\begin{equation}}
\newcommand{\eeq}{\end{equation}}
\newcommand{\beqa}{\begin{eqnarray}}
\newcommand{\eeqa}{\end{eqnarray}}
\newcommand{\beqas}{\begin{eqnarray*}}
\newcommand{\eeqas}{\end{eqnarray*}}
\newcommand{\ba}{\begin{array}}
\newcommand{\ea}{\end{array}}
\newcommand{\bi}{\begin{statementize}}
\newcommand{\ei}{\end{statementize}}

\def\bR{{\mathbb R}}

\title{Adaptive Newton-CG methods with global and local analysis for unconstrained optimization with H\"older continuous Hessian}

\author{
Ziyang Zeng\thanks{Department of Industrial Systems Engineering and Management, National University of Singapore, Singapore (email: {\tt ziyangzeng@u.nus.edu}, {\tt junyuz@nus.edu.sg}). The work of Junyu Zhang was partially supported by the Singapore Ministry of Education Academic Research Fund Tier 2 (MOE-T2EP20125-0007).}
\and
Junyu Zhang$^*$
\and
Chuan He\thanks{Department of Mathematics, Link\"oping University, Sweden (email: {\tt chuan.he@liu.se}). The work of Chuan He was partially supported by the Wallenberg AI, Autonomous Systems and Software Program (WASP) funded by the Knut and Alice Wallenberg Foundation. Corresponding author.}
}

\begin{document}
\maketitle
\begin{abstract}In this paper, we propose adaptive Newton-conjugate gradient (Newton-CG) methods for minimizing a nonconvex function $f$ whose Hessian is $(H_f,\nu)$-H\"older continuous with modulus $H_f>0$ and exponent \(\nu\in(0,1]\). The proposed methods leverage the gradient norm and historical information to adaptively update the regularization parameters in the Newton systems. This strategy enables efficient use of each Newton system solve and automatically reduces the regularization parameter as the iterates approach a minimizer, leading to fast local convergence. Unlike existing globalized Newton-CG algorithms, the proposed approach avoids repeated trial solves of the Newton system to tune regularization parameters and eliminates the slow local convergence caused by fixed regularization parameters. The proposed methods achieve the best-known global iteration complexity ${\mathcal O}\big(H_f^{1/(1+\nu)}\epsilon^{-(2+\nu)/(1+\nu)}\big)$ for finding an $\epsilon$-stationary point, while also enjoying local superlinear convergence near nondegenerate local minimizers. Numerical experiments demonstrate clear practical advantages of the proposed methods.
\end{abstract}






\section{Introduction}
We consider the nonconvex unconstrained optimization problem
\begin{equation}\label{ucpb}
    \min_{x\in \mathbb{R}^n} f(x),
\end{equation}
where \(f:\mathbb{R}^n\to\mathbb{R}\) is twice continuously differentiable. Our focus is on second-order algorithms that compute an \(\epsilon\)-stationary point \(\bar{x}\) satisfying \(\|\nabla f(\bar{x})\|\le \epsilon\), under the assumption that the Hessian is $(H_f,\nu)$-H\"older continuous on a suitable compact set $\mathcal{X}$ that shall be specified later. That is, there exist constants $H_f>0$ and $\nu\in(0,1]$ such that  
$\left\|\nabla^2 f(y) - \nabla^2 f(x)\right\| \le H_f\|y-x\|^\nu$ for all $x,y\in\mathcal{X}$. In the case \(\nu=1\), this condition corresponds to the standard Lipschitz-continuous Hessian assumption, which has been extensively studied in the literature. In contrast, the H\"older-continuous regime \(\nu<1\) has received comparatively less attention. In this work, we consider the full range \(\nu\in(0,1]\).

Under the Lipschitz-Hessian setting \((\nu=1)\), second-order methods are among the most powerful tools for solving \eqref{ucpb} at
small to medium-sized problems, due to their local superlinear convergence near nondegenerate solutions. However, it is well known that the classic Newton method may fail to converge globally. To address this, a popular globalization strategy is the cubic-regularized Newton method \citep{nesterov2006cubic}, which not only enjoys global convergence with an optimal $\mathcal O(\epsilon^{-3/2})$ iteration complexity for nonconvex functions \citep{cartis2018worst,CaDuHiSi20}, but also retains local superlinear convergence. Yet the need to repeatedly solve the cubic subproblems often constitutes a computational bottleneck in practical applications. To mitigate this cost, one may solve the subproblems inexactly with first-order methods \citep{CaDu19}. As a more mature and practically stable globalization strategy, the Levenberg-Marquardt (quadratic) regularization \citep{Levenberg1944Method, Marquardt1963Algorithm} is also widely used, with the resulting linear systems solved using the conjugate gradient (CG) method:
\begin{equation}\label{eq: damped newton}
    \left( \nabla^{2} f(x^k) + \varepsilon_{k} I \right) d^{k} = - \nabla f(x^k),
\end{equation}
where $d^k$ is the the search direction at iteration $k$. The regularization (damping) parameter $\varepsilon_k$ plays a critical role in both theory and performance. For example, by setting $\varepsilon_k \equiv \sqrt{\epsilon}$, i.e., equal to the square root of the target accuracy $\epsilon$, \citet{royer2020newton} proposes a capped CG procedure to solve \eqref{eq: damped newton} and obtain an $\mathcal{O}(\epsilon^{-3/2})$ iteration complexity for finding an $\epsilon$-stationary point. While this strategy achieves the optimal global rates, its practical behavior can be unsatisfactory when the target precision $\epsilon$ is small. Specifically, if $\varepsilon_k\equiv\sqrt{\epsilon}$ is small, the method may repeatedly select negative curvature directions in early iterations, resulting in slow early progress.  Moreover, using a non-adaptive $\varepsilon_k$ can destroy the local superlinear convergence that one expects from second-order methods. To address this issue, \citet{zhou2025regularized}
proposes a gradient-norm-based regularization rule that 
retains both fast global and local convergence. A more complete discussion of the literature can be found in later sections.


As for the H\"older-Hessian regime $\nu<1$, results are still quite limited. Recently, \citet{he2025newton} proposed a framework that sets $\varepsilon_k\equiv\epsilon^{{\nu}/{(1+\nu)}}$ and uses the capped CG to solve the corresponding linear systems. This method finds an $\epsilon$-stationary point within $\mathcal{O}(\epsilon^{-(2+\nu)/(1+\nu)})$ iterations,  matching the lower bound established by \citet{cartis2018worst}. However, the same two limitations persist: (i) a too small and non-adaptive $\varepsilon_k$ can lead to unnecessarily excessive reliance on negative curvature directions in early stages, producing weak descent sometimes even worse than gradient descent, and (ii) the desirable local superlinear convergence near nondegenerate local minimizers is missing. A remedy for this is to adapt the damping parameter \(\varepsilon_k\), at the cost of introducing a nested line-search procedure that is agnostic to the problem parameters. More precisely, one needs to solve \eqref{eq: damped newton} multiple times per iteration in the outer line-search loop for an appropriate \(\varepsilon_k\), while another inner line-search loop is required for each $\varepsilon_k$ to check whether a sufficient descent is obtained. This results in a computationally expensive nested double-loop line-search structure. \vspace{0.2cm}


\noindent\textbf{Contributions. } Motivated by these considerations, our main contributions are highlighted below. 
\begin{itemize}
    \item We propose two Newton-CG methods, depending on the availability of $\nu$, both of which adaptively regularize the Newton system leveraging the current gradient magnitudes. To alleviate the nested double line-search loop issue, we develop novel adaptive mechanisms to estimate the regularization parameters 
    using historical local information, thereby eliminating the outer line-search loop for selecting $\varepsilon_k$. This allows our algorithms to solve \eqref{eq: damped newton} only once per iteration, improving practical efficiency. 
    \item Both methods achieve the best-known iteration complexity
\(
{\mathcal{O}}\big({H_f^{1/(1+\nu)}}{\epsilon^{-(2+\nu)/(1+\nu)}}\big).
\)
For both algorithms, we establish local superlinear convergence near nondegenerate stationary points, which is generally unattainable under non-adaptive damping schemes. A technically interesting point is that when developing a fully parameter-free variant, a direct application of local Hessian modulus estimates
does not work, as these local estimates 
can become unbounded in the H\"olderian regime ($\nu<1$). Furthermore, a capture theorem is derived to assist the establishment of local superlinear rates.
\end{itemize}
\vspace{0.2cm}


\noindent\textbf{Related literature. } Next, let us review the closely related literature on second-order methods that has yet to be discussed. We begin with the representative choices of damping (regularization) parameters in Newton-type methods under the classical Lipschitz-continuous Hessian assumption. We then summarize the more limited body of work that develops second-order methods for \eqref{ucpb} under  H\"older continuity of the Hessian.

\paragraph{\it The Lipschitz-Hessian regime.}
In the classical Lipschitz-Hessian setting \((\nu=1)\) , substantial work has investigated the choice of the damping parameter \(\varepsilon_k\) in the regularized Newton system.
For convex functions, early works \citep{li2004regularized,polyak2009regularized} proposed gradient-based regularization of the form $\varepsilon_k\propto \|\nabla f(x^k)\|$, yielding quadratic local convergence but not providing satisfactory global complexity guarantees. More recently, \citet{mishchenko2021regularized} analyzed the choice $\varepsilon_k\propto \|\nabla f(x^k)\|^{1/2}$, establishing a global $\mathcal{O}(1/k^2)$ rate while successfully preserving local superlinear convergence. In the nonconvex setting, however, beyond the challenge of Hessian indefiniteness, a tension exists between achieving the optimal global rate and maintaining local superlinear convergence. \citet{Ueda2010Convergence} proposed a regularization scheme based on the minimum eigenvalue of the Hessian and the gradient norm. While preserving a local superlinear rate, it yields a suboptimal global complexity of $\mathcal{O}(\epsilon^{-2})$. To improve global performance, \citet{gratton2024yet}  designed a trust-region-based method that alternates between regularized Newton and negative curvature steps, which improves the global rate to nearly optimal. Alternatively, \citet{royer2020newton} fixed $\varepsilon_k$ to the target accuracy and employed the capped CG procedure, achieving the optimal global rate, while Curtis et al.\ \citep{curtis2021trust} further improved this method by leveraging trust-region techniques. However, the non-adaptive regularization does not lead to local superlinear convergence. Along a similar line, \citet{Zhu2024Hybrid}  achieved the optimal global rate together with local superlinear convergence using capped CG under additional error bound or global strong convexity assumptions. These results imply an inherent trade-off: gradient-norm-based regularization generally favors local behavior, as the regularization term vanishes asymptotically ($\|\nabla f(x^k)\|\rightarrow 0$), allowing the algorithm to behave like pure Newton steps near the solution. On the other hand, fixed-accuracy regularization, while globally optimal, often sacrifices this local rate. Very recently, \citet{zhou2025regularized} bridged this gap by introducing novel gradient-norm-based regularization schemes that achieve both the best-known global rate and local superlinear convergence.

\paragraph{\it The H\"older-Hessian regime.} 
Second-order methods for \eqref{ucpb} under a H\"older-continuous Hessian have received relatively limited attention. The main existing approaches in this setting are based on regularized Newton frameworks
\citep{cartis2011adaptive,cartis2019universal,cartis2020sharp,grapiglia2017regularized,he2025newton,zhang2023riemannian}. Specifically, the early work of \citet{cartis2011adaptive} addresses \eqref{ucpb} by solving a sequence of $(2+\nu$)-th regularized subproblems, which can be viewed as a natural generalization of cubic regularization from the Lipschitz-Hessian case. This line was subsequently extended by \citet{cartis2019universal,cartis2020sharp} to high-order regularized methods for nonconvex optimization under H\"older continuity of higher-order derivatives.  In a related direction, \citet{grapiglia2017regularized} tackles \eqref{ucpb} by solving a sequence of $(2+\nu$)-th or cubic regularized Newton subproblems, while \citet{zhang2023riemannian} develops adaptive regularized Newton schemes on Riemannian manifolds (\eqref{ucpb} is a special case) that inexactly solve either \((2+\nu)\)-th order regularized Newton or trust-region subproblems. All these methods achieve optimal worst-case iteration complexity, but they typically require solving higher-order regularized subproblems or nontrivial polynomial optimization problems, which can be prohibitively expensive. Building on the quadratic regularized Newton and capped CG framework of \citet{royer2020newton}, \citet{he2025newton} proposed a parameter-free framework for solving \eqref{ucpb} with optimal global rate. However, its non-adaptive regularization scheme fails to preserve the appealing local superlinear convergence. In fact, to the best of our knowledge, no existing work that simultaneously attains  optimal global rate and guarantees local superlinear rate for nonconvex problem with H\"older-continuous Hessian.  \vspace{0.2cm}

\noindent\textbf{Organization. } In Section~\ref{sec:not-as}, we introduce notation and standing assumptions used throughout the paper. In Sections~\ref{sec:ncg-hdr} and \ref{sec:ncg-hd}, depending on the availability of the H\"older exponent $\nu$, we develop two Newton-CG methods for \eqref{ucpb} that adaptively regularize the Newton system while estimating the $H_f$ in a line-search-free, auto-conditioned manner. Global iteration complexity and local superlinear convergence guarantees are established for both methods. Sections~\ref{sec:num} and \ref{sec:conclusion} present numerical results and conclusion, respectively. Appendix~\ref{apdx} contains the proofs of the main results. Appendix \ref{appendix:capped-CG} briefly introduces the capped CG procedure used in our algorithms.  \vspace{0.2cm}

\noindent\textbf{Notations. } We use $\|\cdot\|$ to denote the $\ell_2$ norm of a vector or the spectral norm of a matrix. We denote the level set of $f$ at $u\in\mathbb{R}^n$ by $\mathscr{L}_f(u):= \{x:f(x)\le f(u)\}$, and we denote its $r$-neighborhood by $\mathscr{L}_f(u;r):= \{x:\|x-x^\prime\|\le r,\ x^\prime\in\mathscr{L}_f(u)\}$ for all $r\ge0$. For any $z\in\mathbb{R}$, we denote the sign function as $\sgn(z)$, which returns $1$ when $z\geq0$ and $-1$ when $z<0$.

\section{Preliminaries and basic assumptions}\label{sec:not-as}
By default, we assume the objective function $f\in\mathcal{C}^2$ to be twice continuously differentiable and has bounded level set, as formalized below.   

\begin{assumption}\label{asp:basic}
The level set $\mathscr{L}_f(x^0)$ at initial iterate $x^0$ is compact.
\end{assumption}
As a direct consequence of Assumption \ref{asp:basic},  there exist $f_{\mathrm{low}}\in\bR$, $U_g>0$ and $U_H>0$ such that 
\begin{equation}\label{def:some-qtt}
f(x)\ge f_{\mathrm{low}},\quad \|\nabla f(x)\|\le U_g,\quad \|\nabla^2 f(x)\|\le U_H\qquad \forall x\in \mathscr{L}_f(x^0).    
\end{equation}
Define $\Delta_f:=f(x^0)-f_{\rm low}$. Also, according to the algorithmic design of this paper, all analysis can be restricted to an $r_d$-neighborhood of the level set $\mathscr{L}_f(x^0)$, with $r_d$ defined as 
\begin{equation}\label{def:rd}
r_d := \max\{1.1, 1.1U_g,U_H\}.
\end{equation}
The factor $1.1$ here can be replaced by any constant greater than $1$. We next make the assumption of H\"older continuity on $\nabla^2f$ over the compact region $\mathscr{L}_f(x^0;r_d)$. 
\begin{assumption}\label{asp:hd}
The Hessian $\nabla^2 f$ is $(H_f,\nu)$-H\"older continuous on $\mathscr{L}_f(x^0;r_d)$ such that 
\begin{equation*}
\|\nabla^2 f(y) - \nabla^2 f(x)\| \le H_f\|y-x\|^\nu\qquad\forall x,y\in\mathscr{L}_f(x^0;r_d),  
\end{equation*}
for some $\nu\in(0,1]$ and $H_f>0$. 
Without loss of generality, we default $H_f\geq1$.
\end{assumption} 
Denote the Taylor residuals of $f$ and $\nabla f$, respectively, as
\begin{equation}\begin{aligned}
\mathcal{R}_0(y,x) & := \Big|f(y)-f(x)-\nabla f(x)^{\rm T}(y-x)-\frac{1}{2}(y-x)^{\rm T}\nabla^2 f(y,x)(y-x)\Big|,\label{def:R0}\\
\mathcal{R}_1(y,x) & := \|\nabla f(y) - \nabla f(x) - \nabla^2 f(x)(y-x)\|.
\end{aligned}
\end{equation}
Then Assumption \ref{asp:hd} immediately implies the following residual bounds \citep{grapiglia2017regularized}:  
\begin{equation}\label{ineq:desc-hd}
\mathcal{R}_0(y,x)\le\frac{H_f\|y-x\|^{2+\nu}}{(1+\nu)(2+\nu)},\qquad \mathcal{R}_1(y,x)\le\frac{H_f\|y-x\|^{1+\nu}}{1+\nu} \qquad \forall x,y\in \mathscr{L}_f(x^0;r_d),    
\end{equation}
As a consequence, we can define two handy \emph{lower} estimators of the constant $H_f$ as 
\begin{equation}
    \label{defn:estimators}
    \mathcal{H}_{0}(y,x):=\frac{2\mathcal{R}_0(y,x)}{\|y-x\|^{2+\nu}},\qquad\mathcal{H}_{1}(y,x):=\frac{\mathcal{R}_1(y,x)}{\|y-x\|^{1+\nu}}\qquad\forall x\neq y.
\end{equation}
It is immediate that $H_f\geq\mathcal{H}_{0}(y,x)$ and $H_f\geq \mathcal{H}_{1}(y,x)$ for all $x,y\in \mathscr{L}_f(x^0;r_d)$ and $x\neq y$.

\section{An adaptive regularized Newton-CG method}\label{sec:ncg-hdr}
In this section, we propose an adaptive Newton-CG method for problem~\eqref{ucpb}, under the basic setting where the Hölder exponent $\nu$ is known. This method uses \eqref{defn:estimators} to estimate the local Hölder constant $H_f$ in $\mathscr{L}_f(x^0; r_d)$, drawing inspiration from the auto-conditioning technique used to develop line-search-free first-order methods (e.g., \citet{lan2024projected, li2025simple}).



\subsection{Algorithm framework} Based on the previous discussion, we present Algorithm~\ref{alg:NCG-pd}. At each iteration $k\geq0$ of this algorithm, we use 
a \texttt{CappedCG} subroutine (Algorithm \ref{alg:capped-CG}, Appendix \ref{appendix:capped-CG}) proposed by \citet{royer2020newton} to solve the damped Newton system \eqref{eq: damped newton}:
\begin{equation}\label{damp-Newton-sys-nu}
\left(\nabla^2 f(x^k) + 2(\gamma_k\|\nabla f(x^k)\|^\nu)^{\frac{1}{1+\nu}} I\right)d = -\nabla f(x^k).
\end{equation}
where $\gamma_k$ is an adaptive estimation of $H_f$. Rather than approximating $H_f$ through a backtracking search at each iteration, we estimate it using the current estimate along with historical estimates of $H_f$ based on \eqref{defn:estimators}. 
It is worth mentioning that all the information for updating the local estimate $\sigma_k$ is based solely on historical information of the algorithm that has already been computed in the previous steps, no extra Hessian/gradient/function evaluations are needed. Then, depending on whether \texttt{CappedCG} returns an approximate solution (SOL) of \eqref{damp-Newton-sys-nu} or a negative curvature (NC) direction of $\nabla^2 f(x^k)$, Algorithm \ref{alg:NCG-pd} will exploit the damped Newton or negative curvature directions, respectively.

\begin{algorithm}[!tbh]
\caption{An adaptive regularized Newton-CG method}
\label{alg:NCG-pd}
\footnotesize
\DontPrintSemicolon

\SetKw{KwRet}{return}

\ShowLn\KwIn{
initial point $x^0\in\mathbb{R}^n$, parameter $\gamma_0\ge1$, line-search parameters $\eta,\theta\in(0,1)$, exponent $\nu\in(0,1]$.}

\ShowLn\While{$\nabla f(x^k) \neq 0$}{
  \ShowLn 
  Set  $H_k=\nabla^2 f(x^k)$, $g_k=\nabla f(x^k)$,
   $\varepsilon_k=(\gamma_k\|g_k\|^\nu)^{1/(1+\nu)}$,
  and $\zeta_k=\min\{1/2,\|g_k\|^{\nu/(1+\nu)}\}$. Then call  
   $$(d,\text{d\_type})\leftarrow\texttt{CappedCG}(H_k,g_k,\varepsilon_k,\zeta_k,U).$$

  \ShowLn\eIf{ \emph{d\_type == NC} }{
    Set $d^k \leftarrow -\sgn(d^{\rm T}g_k)\frac{|d^{\rm T} H_k d|}{\|d\|^3}d$, and find $\alpha_k=\theta^{j_k}$, where $j_k$ is the smallest nonnegative integer $j$ such that
    \begin{equation}\label{ls-nc-stepsize}
      f(x^k+\theta^j d^k) < f(x^k) - \frac{\eta}{2}\theta^{2j}\|d^k\|^3 .
    \end{equation}

    \ShowLn \textbf{if} $j_k\ge 1$ \textbf{then} set $\sigma_k = \mathcal{H}_0(x^k+ \theta^{j_k-1}d^k,x^k)$.
     
  }(\tcp*[f]{$\text{d\_type}=\mathrm{SOL}$}){

    \ShowLn Set $d^k \leftarrow d$, and find $\alpha_k=\theta^{j_k}$, where $j_k$ is the smallest nonnegative integer $j$ such that
    \begin{equation}\label{ls-sol-stepsize}
      f(x^k + \theta^j d^k) < f(x^k) - \eta\varepsilon_k\theta^{j}\|d^k\|^2 .
    \end{equation}
    
    \ShowLn  \textbf{if}  $j_k==0$  \textbf{then} set $\sigma_k = \mathcal{H}_1(x^k+d^k,x^k)$  \textbf{ else }   set  $\sigma_k = \max\left\{\mathcal{H}_0(x^k+d^k,x^k), \mathcal{H}_0(x^k+\theta^{j_k-1}d^k,x^k)\right\}$.
    
  }\vspace{0.1cm}

  \ShowLn Set $x^{k+1} = x^k + \alpha_k d^k$, $\gamma_{k+1}=\max\{\gamma_k,\sigma_k\}$, and $k\leftarrow k+1$.\;
}
\end{algorithm}

As discussed in the introduction, another key distinction, in addition to the adaptive estimation of $H_f$, from the existing parameter-free Newton-CG methods \citep{he2025newton} for problem \eqref{ucpb} is that, instead of setting the damping parameter to $\varepsilon_k \propto \epsilon^{\nu/(1+\nu)}$, we set 
$\varepsilon_k\propto\|\nabla f(x^k)\|^{\nu/(1+\nu)}$. On the one hand, in early iterations when the gradients are still large, the algorithm is more likely to exploit the scaled gradient steps, suppressing the negative curvature steps that are usually slower in early stages. On the other hand, when the iterates get close to a nondegenerate local solution of the problem, the adaptively selected $\varepsilon_k$ automatically diminishes as gradient decreases. By setting the relative accuracy of \texttt{CappedCG} to $\zeta_k = O\big(\|\nabla f(x^k)\|^{\nu/(1+\nu)}\big)$, a local superlinear convergence can also be expected. 

\subsection{Global complexity bound} The following lemma shows that the main iterates and trial iterates of Algorithm \ref{alg:NCG-pd} lie within a suitable neighborhood of the level set, with the proof given in Appendix~\ref{ssec: pf-sec3}. 
\begin{lemma}\label{lem:trial-xk-ccg}
Given Assumption \ref{asp:basic}, the sequences $\{x^k\}_{k\geq0}$ and $\{d^k\}_{k\geq0}$ generated by Algorithm~\ref{alg:NCG-pd} satisfy $x^k+\alpha d^k\in\mathscr{L}_f(x^0;r_d)$ for all $k\geq0$ and $\alpha\in[0,1]$, where $r_d$ is defined in \eqref{def:rd}.
\end{lemma}
This lemma informs us that all analysis can be performed under the constants introduced in \eqref{def:some-qtt} and we can activate Assumption \ref{asp:hd} to utilize the $(H_f,\nu)$-H\"older continuity of $\nabla^2f$ in $\mathscr{L}_f(x^0;r_d)$. 

Next, we proceed with the global complexity analysis for finding $\epsilon$-stationary points of problem \eqref{ucpb}. As the the estimating mechanism based on \eqref{defn:estimators} 
always \emph{underestimates} $H_f$, a sufficient decrease is not necessarily guaranteed even if \texttt{CappedCG} successfully returns an approximate solution to \eqref{damp-Newton-sys-nu}, a standard analysis may therefore fail. For Algorithm \ref{alg:NCG-pd}, we divide its iterations before reaching an $\epsilon$-stationary point 
($\mathbb{K}_\epsilon:=\{k:\|\nabla f(x^t)\|>\epsilon, \forall t\leq k\}$) into three subsets:
\begin{equation*}\begin{aligned}
\mathbb{K}_{\epsilon,1} & := \{k\in\mathbb{K}_\epsilon:\|\nabla f(x^{k+1})\|\ge \|\nabla f(x^k)\|/2,\ \sigma_k \le 2\gamma_k\},\\
\mathbb{K}_{\epsilon,2} & := \{k\in\mathbb{K}_\epsilon:\|\nabla f(x^{k+1})\|\ge \|\nabla f(x^k)\|/2,\ \sigma_k > 2\gamma_k\},\\
\mathbb{K}_{\epsilon,3} & := \{k\in\mathbb{K}_\epsilon:\|\nabla f(x^{k+1})\|< \|\nabla f(x^k)\|/2\},
\end{aligned}\end{equation*}
where iterations in $\mathbb{K}_{\epsilon,1}$ generate sufficient descent, while $|\mathbb{K}_{\epsilon,2}|$ and $|\mathbb{K}_{\epsilon,3}|$ are provably small. In the following, we provide two lemmas that establish the sufficient descent for $\mathbb{K}_{\epsilon,1}$, depending on the output of \texttt{CappedCG}. The proofs of the two lemmas are deferred to Appendix~\ref{ssec: pf-sec3}.

\begin{lemma}\label{lem:desc-sol-lip}
Given Assumptions \ref{asp:basic} and \ref{asp:hd}, for all $k\in\mathbb{K}_{\epsilon,1}$ with $d^k$ being the output of \emph{\texttt{CappedCG}} with d$\_$type=SOL, the following two statements hold.
\begin{enumerate}
\item[{\rm (i)}] The step size $\alpha_k$ is well-defined and it satisfies $\alpha_k\ge \min\big\{1,(2(1-\eta)/(1.1^\nu H_f))^{1/(1+\nu)}\theta\big\}$.

\item[{\rm (ii)}] Let $c_{\mathrm{sol},\nu}:= {\eta(1-\eta)^{2/\nu}\theta}/{100}$. The next iterate $x^{k+1}=x^k+\alpha_k d^k$ satisfies 
\begin{equation}\label{lwbd:f-descent}
f(x^k) - f(x^{k+1}) \ge c_{\mathrm{sol},\nu}\gamma_k^{-\frac{1}{1+\nu}}\|\nabla f(x^k)\|^{\frac{2+\nu}{1+\nu}}.
\end{equation}
\end{enumerate}
Moreover, the gradient at the next iterate is bounded by $\|\nabla f(x^{k+1})\| \leq (2H_f+5) \|\nabla f(x^k)\|$.    
\end{lemma} 

\begin{lemma}\label{lem:desc-nc-lip}
Given Assumptions \ref{asp:basic} and \ref{asp:hd}, for all $k\in\mathbb{K}_{\epsilon,1}$ with $d^k$ being the output of \emph{\texttt{CappedCG}} with d$\_$type=NC, the following two statements hold.
\begin{enumerate}
    \item[{\rm (i)}] The step size $\alpha_k$ is well-defined, and $\alpha_k\ge\min\{1,\theta((1-\eta)/H_f)^{1/\nu}\|\nabla f(x^k)\|^{(1-\nu)/(1+\nu)}\}$.
    \item[{\rm (ii)}] Let $c_{\mathrm{nc},\nu}:= \eta(1-\eta)^{\frac{2}{\nu}}\theta^2/2^{\frac{2+\nu}{\nu}}$. The next iterate $x^{k+1} = x^k + \alpha_k d^k$ satisfies 
    \begin{equation}\label{ineq:nc-suff-desc}
        f(x^k) - f(x^{k+1}) \ge c_{\mathrm{nc},\nu} \gamma_k^{-\frac{1}{1+\nu}}\min\left\{\|\nabla f(x^k)\|^{\frac{2+\nu}{1+\nu}},1\right\}.
    \end{equation} 
\end{enumerate}
Moreover, we have $\|\nabla f(x^{k+1})\| \leq 2\alpha_kU_H\|d^k\|$ whenever $\|d^k\|\leq M\!:=\!\min\big\{U_H^\nu,\theta(1-\eta)^{\frac1\nu}U_H/H_f^{\frac{1}{\nu}}\big\}$.
\end{lemma}

In either case, an $\Omega(\epsilon^{(2+\nu)/(1+\nu)})$ descent can be achieved for $k\in\mathbb{K}_{\epsilon,1}$ whenever $\|\nabla f(x^k)\|\geq\epsilon$. Recall that $\Delta_f=f(x^0)-f_{\rm low}$ denotes the function value gap. It then follows directly that $|\mathbb{K}_{\epsilon,1}|\leq \mathcal{O}(\Delta_f\epsilon^{-(2+\nu)/(1+\nu)})$. Since the sequence $\{\gamma_k\}_{k\ge0}\subseteq[\gamma_0,H_f]$ is nondecreasing in Algorithm \ref{alg:NCG-pd}, it is also straightforward to bound $|\mathbb K_{\epsilon,2}|\leq \lceil\log_2(H_f/\gamma_0)\rceil = \mathcal{O}(1)$ as $\gamma_k$ doubles for each $k\in\mathbb{K}_{\epsilon,2}$. Note that $\mathbb{K}_{\epsilon,3}$ can be divided into at most $|\mathbb{K}_{\epsilon,1}|+|\mathbb{K}_{\epsilon,2}|+1$ consecutive subsets with each containing at most $\lceil\log_2(U_g/\epsilon)\rceil$ iterations as the gradient halves for each $k\in\mathbb{K}_{\epsilon,3}$. This straightforward analysis immediately yields an iteration complexity of $\mathcal{O}\big(\epsilon^{-(2+\nu)/(1+\nu)}\ln(1/\epsilon)\big)$. In the next theorem, we show that the logarithmic factor $\ln(1/\epsilon)$ can be removed with a more careful analysis; see the detailed proof in Appendix~\ref{appdx:Thm-1}. 
\vspace{-0.2cm}
\begin{theorem}\label{thm:1st-cmplx-ancg}
Given Assumptions \ref{asp:basic} and \ref{asp:hd}, we have $|\mathbb{K}_\epsilon|\leq\mathcal{O}\big(\Delta_fH_f^{\frac{1}{1+\nu}}\epsilon^{-\frac{2+\nu}{1+\nu}}\big)$ for Algorithm~\ref{alg:NCG-pd}. 
\end{theorem}
The iteration complexity of Theorem \ref{thm:1st-cmplx-ancg} matches the best known upper bound of \citet{he2025newton} (non-adaptive damping scheme), and it also matches the lower bound provided by \citet{cartis2018worst}. In the Lipschitz-Hessian special case ($\nu=1$), the iteration complexity of Theorem \ref{thm:1st-cmplx-ancg} reduces to  
\(\cO\big(\Delta_fH_f^{1/2}\epsilon^{-3/2}\big)\), outperforming the $\cO\big(\Delta_fH_f^{2}\epsilon^{-3/2}\big)$ bound by \citet{hlp2023ncgal,royer2020newton,Zhu2024Hybrid}, {and matching the best known upper bounds of \citet{zhou2025regularized}}.

\subsection{Local superlinear convergence}
By choosing $\varepsilon_k \propto\|\nabla f(x^k)\|^{\nu/(1+\nu)}$, Algorithm \ref{alg:NCG-pd} allows that $\varepsilon_k\to0$ as $\|\nabla f(x^k)\|\to0$. Then, automatically, the damped Newton system \eqref{damp-Newton-sys-nu} asymptotically reduces to the exact Newton system as the algorithm converges to a nondegenerate solution $x^*$, indicating local superlinear convergence. In the next theorem, we formally state this result. 

\begin{theorem}
\label{thm:loc-conv-1}
Given Assumptions \ref{asp:basic} and \ref{asp:hd}, let $x^*$ be an arbitrary nondegenerate local minimizer of $f$ such that $\nabla^2 f(x^*) \succeq \mu I$ for some $\mu > 0$. Then there exists $\delta>0$ such that, for the sequence $\{x^k\}_{k\geq0}$ generated by Algorithm \ref{alg:NCG-pd}, if $x^{k_0}\in B_{\delta}(x^*)$ for some $k_0\geq0$, then $\{x^k\}_{k\geq k_0} \subseteq B_{\delta}(x^*)$ and $\{x^k\}_{k\geq k_0}$ converges to $x^*$ superlinearly in the sense that 
$\|x^{k+1}-x^*\|\leq O(\|x^k-x^*\|^{(1+2\nu)/(1+\nu)})$.
\end{theorem}

The key idea in the analysis is to show that the \texttt{CappedCG} subroutine will always successfully return an approximate solution to the damped Newton system (d\_type = SOL) with $\alpha_k=1$. Once the iterates $x^k$ are sufficiently close to $x^*$, then the remaining analysis will be standard. The detailed proof of this theorem is moved to Appendix \ref{appdx:Thm-local-Known-Nu}. To our best knowledge, this is the first analysis for Newton-CG methods that simultaneously attains optimal global complexity and local superlinear convergence under the H\"older-Hessian setting, highlighting the advantage of adaptive damping as opposed to non-adaptive damping schemes like the one in \cite{he2025newton}. 

\section{A universal adaptive regularized Newton-CG method}\label{sec:ncg-hd}
Notice that Algorithm~\ref{alg:NCG-pd} requires prior knowledge of the H\"older exponent $\nu$ to determine the damping parameter. Although $\nu$ is often known in many applications, it is still desirable to have a completely parameter-free method that is universally optimal for all $\nu\!\in\!(0,\!1]$, which will be the focus of this section.

\subsection{Algorithm framework}
Due to the lack of $\nu$, we cannot form the $\nu$-dependent adaptive damping scheme \eqref{damp-Newton-sys-nu}. Instead, we view H\"older continuity as an approximate Lipschitz continuity with controllable error, whose underlying idea can be traced back to \citet{devolder2014first,nesterov2015universal}.  

\begin{proposition}[{\citet[Lemmas~1--2]{he2025newton}}]\label{prop:inexact_oracle}
Given Assumption~\ref{asp:hd}, for any $z>0$ and $a\ge 2$, we have
\begin{equation}\label{inexact-oracle}
\mathcal{R}_1(y,x)\;\le\;\frac{a\,\gamma_\nu(z)}{16}\,\|y-x\|^2\;+\;\frac{z}{a}
\qquad \forall\, x,y\in\mathscr{L}_f(x^0;r_d),
\end{equation}
where the inexact Lipschitz constant is a function of the error level 
$\gamma_\nu(z) := 4H_f^{\frac{2}{1+\nu}}z^{-\frac{1-\nu}{1+\nu}}$.
\end{proposition}

We remark that the function $\gamma_\nu(\cdot)$ is used only in the analysis and is not required in the actual implementation of the algorithm. As detailed in Algorithm \ref{alg:NCG-hd}, we solve
\begin{equation*}
\Big(\nabla^2 f(x^k) + 2\big(\gamma_k\|\nabla f(x^k)\|\big)^{1/2} I\Big)d = -\nabla f(x^k) 
\end{equation*}
by pretending $\nu=1$ in the subproblem \eqref{damp-Newton-sys-nu} of Algorithm \ref{alg:NCG-pd}. The sequence $\{\gamma_k\}$ adaptively estimates $\{\gamma_\nu(\|\nabla f(x^k)\|)\}$. However, we should also note an important difference that, unlike Algorithm \ref{alg:NCG-pd} where $\{\gamma_k\}$ is always upper bounded by $H_f$, the inexact Lipschitz constant $\gamma_\nu(\|\nabla f(x^k)\|)\to+\infty$ when $\|\nabla f(x^k)\|\to0$, posting new challenges in the analysis of the algorithm. Also, due to such potential unboundedness, Algorithm~\ref{alg:NCG-hd} may fail to guarantee sufficient descent. When this occurs, we increase $\gamma_k$ by a fixed factor, instead of using local smoothness modulus estimation as adopted in Lines 6 and 9 of Algorithm~\ref{alg:NCG-pd}.

\begin{algorithm}[!tbh]
\caption{A universal adaptive regularized Newton-CG method}
\label{alg:NCG-hd}
\footnotesize
\DontPrintSemicolon

\SetKw{KwRet}{return}

\ShowLn\KwIn{
initial point $x^0\in\mathbb{R}^n$, parameter $\gamma_0\ge1$, line-search parameters $\eta\in(0,{ 1/2}]$ and $\theta\in(0,1)$.} 

\ShowLn Set $c_{\mathrm{sol}} = \frac{\eta(1-\eta)\theta}{400}$.

\ShowLn\While{$\nabla f(x^k)\neq0$}{ 
  \ShowLn Set  $H_k=\nabla^2 f(x^k)$, $g_k=\nabla f(x^k)$,
   $\varepsilon_k=(\gamma_k\|g_k\|)^{1/2}$, and 
   $\zeta_k=\min\big\{1/2,\|g_k\|^{1/2}\big\}$. Then call  
   $$(d,\text{d\_type})\leftarrow\texttt{CappedCG}(H_k,g_k,\varepsilon_k,\zeta_k).$$

  \ShowLn\eIf{\emph{d\_type==NC}}{
    \ShowLn Set
    \(
      d^k \leftarrow
      -\sgn(d^{\rm T}\nabla g_k)\frac{|d^{\rm T}H_k d|}{\|d\|^3}d
    \), and find $\alpha_k=\theta^{j_k}$, where $j_k$ is the smallest nonnegative integer $j$ such that
    \begin{equation}\label{ls-nc-stepsize-hd}
      f(x^k+\theta^j d^k) < f(x^k) - \frac{\eta}{2}\theta^{2j}\|d^k\|^3 .
    \end{equation} 
    
    \ShowLn\textbf{If} {$\|\nabla f(x^k + \alpha_k d^k)\| > \|g_{k}\|/2$ \textbf{and}
            $\alpha_k<\theta/\gamma_k$} \textbf{then} set $\gamma_{k+1}=2\gamma_k$.\;
  
  }(\tcp*[f]{$\text{d\_type}=\mathrm{SOL}$}){

    \ShowLn Set $d^k \leftarrow d$.\;
    \ShowLn \textbf{if} $f(x^k + d^k)\le f(x^k)$ \textbf{and} $\|\nabla f(x^k + d^k)\| \le \|g_k\|/2$ \textbf{then}  set $\alpha_k=1$. \;
    \textbf{else} find $\alpha_k=\theta^{j_k}$, where $j_k$ is the smallest nonnegative integer $j$ such that
      \begin{equation}\label{ls-sol-stepsize-hd-}
        f(x^k + \theta^j d^k) < f(x^k) - \eta\varepsilon_k^{1/2}\theta^j\|d^k\|^2 .
      \end{equation} 
      
    \ShowLn \textbf{if}{ $\|\nabla f(x^k + \alpha_k d^k)\| > \|g_{k}\|/2$ \textbf{and}
            $f(x^k)-f(x^k + \alpha_k d^k)<c_{\mathrm{sol}}\gamma_k^{-1/2}\|g_k\|^{3/2}$}{
    \textbf{then} set $\gamma_{k+1}=2\gamma_k$.\;
  }
  }
\ShowLn Set $x^{k+1} = x^k + \alpha_k d^k$ and $k\leftarrow k+1$.\;
}
\end{algorithm}

\subsection{Global complexity bound} Similar to Lemma \ref{lem:trial-xk-ccg}, the main iterates and trial iterates of Algorithm \ref{alg:NCG-hd} also remain in a properly bounded region, as shown below. 
\begin{lemma}\label{lem:trial-xk-ccg-2}
Given Assumption \ref{asp:basic}, the sequences $\{x^k\}_{k\geq0}$ and $\{d^k\}_{k\geq0}$ generated by Algorithm~\ref{alg:NCG-hd} satisfy $x^k+\alpha d^k\in\mathscr{L}_f(x^0;r_d)$ for all $k\geq0$ and $\alpha\in[0,1]$, where $r_d$ is defined by \eqref{def:rd}.
\end{lemma}
The proof of this lemma is identical to that of Lemma \ref{lem:trial-xk-ccg} and is therefore omitted. In light of Lemma \ref{lem:trial-xk-ccg-2}, we carry out the global complexity analysis of Algorithm \ref{alg:NCG-hd} under the $(H_f,\nu)$-H\"older continuity of $\nabla^2f$ in $\mathscr{L}_f(x^0;r_d)$, as assumed in Assumption \ref{asp:hd}. Again, we partition the  iterations of Algorithm \ref{alg:NCG-hd} before finds an $\epsilon$-stationary point  ($\mathbb{K}_\epsilon:=\{k: \|\nabla f(x^t)\|>\epsilon, \forall  t\le k\}$) into three parts: 
\begin{equation*}\begin{aligned}
\mathbb{K}_{\epsilon,1} & := \big\{k\in\mathbb{K}_\epsilon:\|\nabla f(x^{k+1})\|> \|\nabla f(x^k)\|/2,\ \gamma_k>\gamma_{\nu}(\|\nabla f(x^k)\|)\big\},\\
\mathbb{K}_{\epsilon,2} & := \big\{k\in\mathbb{K}_\epsilon:\|\nabla f(x^{k+1})\|> \|\nabla f(x^k)\|/2,\ \gamma_k\le \gamma_{\nu}(\|\nabla f(x^k)\|)\big\},\\
\mathbb{K}_{\epsilon,3} & := \big\{k\in\mathbb{K}_\epsilon:\|\nabla f(x^{k+1})\|\leq \|\nabla f(x^k)\|/2\big\}.
\end{aligned}\end{equation*}
Among the three sets, the iterations in $\mathbb{K}_{\epsilon,1}$ generate sufficient descent. To upper bound the cardinality of this subset, we establish sufficient descent when \texttt{CappedCG} outputs directions with d$\_$type=SOL, and NC, respectively. The proofs of the two lemmas are deferred to Appendix~\ref{ssec: pf-sec4}.

\begin{lemma}\label{lem:uni-sol-suf-desc}
Given Assumptions \ref{asp:basic} and \ref{asp:hd}, for all $k\in\mathbb{K}_{\epsilon,1}$ with $d^k$ being the output of \emph{\texttt{CappedCG}} with d$\_$type=SOL, the following two statements hold.
\begin{enumerate}
\item[{\rm (i)}] The step length $\alpha_k$ is well defined and it satisfies $\alpha_k\ge(1-\eta)\theta/3$.
\item[{\rm (ii)}] Let $c_{\mathrm{sol}}$ be defined in Algorithm \ref{alg:NCG-hd}. The next iterate $x^{k+1} = x^k + \alpha_k d^k$ satisfies 
\begin{equation}\label{ineq:uni-sol-suf-desc}
f(x^k) - f(x^{k+1}) \ge c_{\mathrm{sol}} \gamma_k^{-\frac{1}{2}}\|\nabla f(x^k)\|^{\frac{3}{2}}.
\end{equation}
\end{enumerate}
Moreover, the next gradient satisfies $\|\nabla f(x^{k+1})\| \leq 5\|\nabla f(x^k)\|$.
\end{lemma}


\begin{lemma}\label{lem:uni-nc-suf-desc}
Given Assumptions \ref{asp:basic} and \ref{asp:hd}, for all $k\in\mathbb{K}_{\epsilon,1}$ with $d^k$ being the output of \emph{\texttt{CappedCG}} with d$\_$type=NC, the following two statements hold.
\begin{enumerate}[{\rm (i)}]
    \item The step length $\alpha_k$ is well-defined, and $\alpha_k\ge\theta/\gamma_k$.
    \item Let $c_{\mathrm{nc}}:= \eta\theta^2/2$. The next iterate $x^{k+1}=x^k+\alpha_k d^k$ satisfies 
    \begin{equation}\label{lwbd:f-descent-hd-nc}
        f(x^k) - f(x^{k+1}) \ge c_{\mathrm{nc}} \gamma_k^{-\frac{1}{2}}\|\nabla f(x^k)\|^{\frac{3}{2}}. 
    \end{equation} 
\end{enumerate}    
Moreover, we have \(\|\nabla f(x^{k+1})\| \leq 2\alpha_kU_H\|d^k\|\) whenever $\|d^k\|\leq \theta U_H$.
\end{lemma}

As discussed above, the sequence $\{\gamma_k\}$ may be unbounded. Consequently, the order (w.r.t. $\epsilon$) of the  $\mathcal{O}(\gamma_k^{-1/2}\|\nabla f(x^k)\|^{3/2})$ descent is not immediately clear from the two lemmas above. The next lemma provides a guaranteed upper bound on $\{\gamma_k\}$, whose proof can be found in Appendix~\ref{ssec: pf-sec4}.

\begin{lemma}\label{lem:upbd-gammak}
Suppose that Assumptions \ref{asp:basic} and \ref{asp:hd} hold. Let $\{\gamma_k\}$ be generated by Algorithm \ref{alg:NCG-hd}. Then, $\gamma_k\le \max\{\gamma_0,2\gamma_\nu(\epsilon)\}$ for all $k\in{\mathbb{K}_\epsilon}$.
\end{lemma}

With the sufficient descent established above, the bound on $|\mathbb K_{\epsilon,1}|$ can be readily obtained. We now provide a brief proof overview for deriving bounds on $|\mathbb K_{\epsilon,2}|$ and $|\mathbb K_{\epsilon,3}|$, respectively. A key observation for bounding $|\mathbb K_{\epsilon,2}|$ is that $\{\gamma_k\}$ is updated only when $k\in\mathbb K_{\epsilon,2}$, but not every iteration in $\mathbb K_{\epsilon,2}$ triggers an update of $\{\gamma_k\}$. Thus, we partition $\mathbb K_{\epsilon,2}$ into two subsets, depending on whether $\gamma_k$ is updated. The number of iterations when $\gamma_k$ is updated can be bounded using the monotonicity of $\{\gamma_k\}$, while the iterations when $\gamma_k$ is not updated can be bounded using sufficient descent. On the other hand, similar to Theorem \ref{thm:1st-cmplx-ancg}, $|\mathbb K_{\epsilon,3}|$ can be bounded by exploiting the decrease of $\{\|\nabla f(x^k)\|\}$. The global iteration complexity of Algorithm \ref{alg:NCG-hd} is stated in the next theorem, whose proof is deferred to Appendix \ref{appdx:Thm-2}.

\vspace{-2mm}
\begin{theorem}\label{thm:glb-cplx-2}
Given Assumptions \ref{asp:basic} and \ref{asp:hd}, then $|\mathbb{K}_\epsilon|\leq\mathcal{O}\big(\Delta_fH_f^{\frac{1}{1+\nu}}\epsilon^{-\frac{2+\nu}{1+\nu}}\big)$ holds for Algorithm~\ref{alg:NCG-hd}.   
\end{theorem}

Theorem \ref{thm:glb-cplx-2} shows that the iteration complexity of Algorithm \ref{alg:NCG-hd} matches the best-known upper bound \citep{he2025newton} and the lower bound \citep{cartis2018worst}. In the Lipschitz-Hessian case ($\nu=1$), it also matches  the best-known upper bound \citep{zhou2025regularized}, improving the dependence on $H_f$ compared to the work of \citet{hlp2023ncgal}, \citet{royer2020newton}, and \citet{Zhu2024Hybrid}.

\subsection{Local superlinear convergence}
Based on the global analysis of Algorithm \ref{alg:NCG-hd}, $\{\gamma_k\}$ may grow unbounded as the gradient vanishes, which precludes the standard local analysis. In this section, we first show that such potential unboundedness is a consequence of degeneracy. 
If Algorithm \ref{alg:NCG-hd} enters a certain neighborhood of a nondegenerate stationary point $x^*$, $\{\gamma_k\}$ will stop growing and hence a common upper bound for $\{\gamma_k\}$ exists for all sufficiently large $k$, enabling the analysis of fast local superlinear convergence.


\begin{lemma}
    \label{lemma:gamma_bound_universal}
    Suppose that Assumptions \ref{asp:basic} and \ref{asp:hd} hold. Let $x^*$ be an arbitrary nondegenerate local minimizer of $f$ such that $\nabla^2 f(x^*) \succeq \mu I$ for some $\mu > 0$. Then there exists $\delta>0$ 
    such that, if $x^k\in B_\delta(x^*)$, then \texttt{CappedCG} will return $\alpha_k=1$ and d\_type = SOL, and Algorithm \ref{alg:NCG-hd} will enter the next iteration with $\gamma_{k+1} = \gamma_k.$
\end{lemma}

This lemma shows that if an iterate falls in a $\delta$-neighborhood of $x^*$, the parameter $\gamma_k$ will stop growing in this step. However, it does not immediately guarantee that the next iterate will still stay in this neighborhood. In fact, the iterates may still, possibly, cross over the boundary of $B_\delta(x^*)$ and return an exploding sequence of $\{\gamma_k\}$.  This causes a ``chicken-and-egg" difficulty in the standard local superlinear analysis: (i) the existence of a non-expansive region requires a finite upper bound on $\{\gamma_k\}$; (ii) the upper bound on $\{\gamma_k\}$ relies on the existence of a non-expansive region (in $B_\delta(x^*)$). Therefore, the standard analysis strategy represented by Theorem \ref{thm:loc-conv-1} no longer works. Inspired by the Capture Theorem of \citet{bertsekas:nonlinear:2016}, we show in the next lemma that a level-set based capture region of the form $B_\delta(x^*)\cap\{x:f(x)\leq f(x^*)+c\},c>0$ exists, so that once the iterates of a descent method enter this region, with the updates bounded by gradient magnitudes, all future iterates will be captured by this region due to a function value barrier. 

\begin{lemma}
    \label{lemma:capture}
    Under the setting of Lemma \ref{lemma:gamma_bound_universal}, 
    there exists a subset $S\subseteq B_\delta(x^*)$ such that   if $x^{k_0}\in S$ for some $k_0\geq0$, then all future iterates will stay in $S$.
\end{lemma}

Combining the above two lemmas, we observe that there exists a neighborhood $S$ of $x^*$ such that once an iterate of Algorithm \ref{alg:NCG-hd} enters $S$, then all future iterates remain in $S$ and the $\gamma_k$ remains a finite constant, and consequently, a local superlinear rate can be established. The proof of Lemma \ref{lemma:gamma_bound_universal} and \ref{lemma:capture}, as well as Theorem \ref{thm:loc-conv-2} are all deferred to Appendix~\ref{ssec: pf-sec4}.

\begin{theorem}
\label{thm:loc-conv-2}
Suppose the conditions of Lemma \ref{lemma:gamma_bound_universal} hold. Let $S$ be defined as in Lemma \ref{lemma:capture}. For $\{x^k\}_{k\geq0}$ generated by Algorithm \ref{alg:NCG-hd}, if $x^{k_0}\in S$ for some $k_0\geq0$, then $\{x^k\}_{k\geq k_0} \subseteq S$ and $\{x^k\}_{k\geq k_0}$ converges to $x^*$ superlinearly in the sense that 
$\|x^{k+1}-x^*\|\leq O(\|x^k-x^*\|^{\min\{1+\nu,\frac{3}{2}\}})$.
\end{theorem}


Note that the local convergence rate of Algorithm~\ref{alg:NCG-hd} is at least as fast as that of Algorithm~\ref{alg:NCG-pd}. This gap follows from the orders of the damping terms, i.e., $\frac{1}{2}\geq \frac{\nu}{1+\nu}$ for $\nu\in(0,1]$.  Intuitively, when an iterate $x^k$ is close enough to a nondegenerate stationary point with $\|\nabla f(x^k)\|\ll 1$, a larger exponent yields a smaller damping magnitude. Consequently, the corresponding linear system is closer to the classical Newton system, which leads to a faster local rate.

\section{Numerical results}\label{sec:num}
We conduct numerical experiments to evaluate the performance of our universal adaptive regularized Newton-CG method (Algorithm~\ref{alg:NCG-hd}, abbreviated as ANCG), and compare it with a parameter-free Newton-CG (abbreviated as HNCG) \citep[Algorithm 2]{he2025newton} and an adaptive cubic regularized Newton method (abbreviated as ACRN) \citep[Universal Method II]{grapiglia2017regularized}. The code to reproduce our numerical results in this section is available at \url{https://github.com/Zeng-Ziyang/ANCG-Holder-Hessian/}.

For all methods, we initialize with $x^0=(1,\ldots,1)^{\rm T}$, and choose the following parameter settings,  
{which provide numerically stable and efficient performance in practice:}
\begin{itemize}
\item For ANCG, we set $(\gamma_{0},\theta,\eta)=(10,0.5,0.01)$;
\item For HNCG, we set $(\gamma_{-1},\theta,r,\eta)=(10,0.5,2,0.01)$;
\item For ACRN, we set $H_0=10$. To solve its cubic regularized subproblems, we employ the gradient descent approach suggested by \citet{CaDu19}, with the initial point uniformly selected from the unit sphere.
\end{itemize}

All the algorithms are coded in Matlab,  and all the computations are performed on a laptop with a 2.60 GHz Intel Core i5-14500 processor and 16 GB of RAM.

\subsection{Infeasibility detection problem}
Consider the infeasibility detection model of \citet{ByCuNo10}:
\vspace{-2mm}
\begin{equation*}
\min_{x\in\mathbb{R}^{n}} \frac{1}{m} \sum_{i=1}^{m} \left( x^{\rm T} A_{i} x + b_{i}^{\rm T} x + c_{i} \right)_{+}^{p},    
\end{equation*}
where $p> 2$, $A_i\in\bR^{n\times n}$, $b_i\in\bR^n$, and $c_i\in\bR$ for $1\le i\le m$. For each triple $(n,m,p)$, we generate 10 random instances and aim to compute a $10^{-4}$-stationary point using HNCG, ANCG, and ACRN.

Table~\ref{table:feas-dect} reports the average runtime, the average number of subproblems solved, and the average number of Hessian-vector products for the three methods. A “subproblem’’ refers to a cubic subproblem for ACRN and to a damped Newton system for HNCG and ANCG.

Overall, the results indicate that ANCG consistently achieves the lowest runtime across all tested dimensions, often reducing runtime by more than half relative to HNCG and by a even larger margin relative to ACRN. Moreover, ANCG requires substantially fewer subproblems and Hessian-vector products, demonstrating improved computational efficiency without compromising the quality of the final solution. 

\begin{table}[htbp]
\centering
\resizebox{\textwidth}{!}{
\begin{tabular}{ccc||ccc||ccc||ccc}
\hline
\multicolumn{3}{c||}{Dimension} 
& \multicolumn{3}{c||}{Runtime (seconds)} 
& \multicolumn{3}{c||}{Total subproblems} 
& \multicolumn{3}{c}{Hessian-vector products} \\
$n$ & $m$ & $p$ 
& HNCG & ANCG & ACRN 
& HNCG & ANCG & ACRN 
& HNCG & ANCG & ACRN \\ 
\hline
100  & 10  & 2.25 & 0.03 & 0.01 & 1.37 & 49.3 & 17.0 & 44.0 & 1567.4 & 405.7 & 707717.9 \\
100  & 10  & 2.50 & 0.02 & 0.01 & 1.44 & 50.6 & 21.0 & 43.8 & 1840.0 & 558.9 & 730280.6 \\
100  & 10  & 2.75 & 0.02 & 0.01 & 1.54 & 54.0 & 24.2 & 49.0 & 2045.4 & 696.4 & 793014.5 \\
100  & 10  & 3.00 & 0.02 & 0.01 & 2.07 & 56.6 & 28.0 & 50.2 & 2221.6 & 833.9 & 872891.8 \\
500  & 50  & 2.25 & 1.30 & 0.83 & 20.47 & 63.0 & 20.0 & 54.3 & 4190.4 & 841.3 & 813247.3 \\
500  & 50  & 2.50 & 1.62 & 0.97 & 19.89 & 72.7 & 25.0 & 56.4 & 5120.6 & 1170.5 & 902732.9 \\
500  & 50  & 2.75 & 1.73 & 1.14 & 19.29 & 72.8 & 30.0 & 62.4 & 5758.3 & 1511.1 & 797625.4 \\
500  & 50  & 3.00 & 1.85 & 1.26 & 14.89 & 75.0 & 34.0 & 65.5 & 6255.3 & 1746.5 & 629736.3 \\
1000 & 100 & 2.25 & 14.94 & 8.28 & 84.02 & 70.1 & 21.0 & 58.0 & 5403.1 & 1013.7 & 874021.2 \\
1000 & 100 & 2.50 & 18.16 & 10.59 & 88.39 & 80.3 & 27.0 & 61.3 & 7078.0 & 1569.0 & 941388.9 \\
1000 & 100 & 2.75 & 20.49 & 12.37 & 62.12 & 84.0 & 32.0 & 64.0 & 7917.6 & 1919.5 & 611309.0 \\
1000 & 100 & 3.00 & 22.07 & 14.19 & 59.87 & 85.1 & 37.0 & 67.7 & 8595.4 & 2295.5 & 552507.6 \\
\hline
\end{tabular}
}
\caption{Comparison of HNCG, ANCG, and ACRN across different dimensions for the infeasibility detection problem.}
\label{table:feas-dect}
\end{table}

\subsection{Single-layer neural networks problem}
We consider the problem of training single-layer RePU neural networks \citep{LiTaYu19RePU}:
\[
\min_{x\in\mathbb{R}^{n}} \frac{1}{m}\sum_{i=1}^{m} \phi\left( (a_{i}^{\rm T}x)_{+}^{p} - b_{i} \right),
\]
where $\phi(t)=t^{2}$ and $p>2$. For each triple $(n,m,p)$, we generate 10 random instances by independently sampling  $a_{i}\sim \mathcal N(0,I_n)$,  and setting $b_{i}=|\bar b_{i}|$ where $\bar b_{i}\sim \mathcal N(0,1)$.  

Our goal is to obtain a $10^{-4}$-stationary point. The numerical results are summarized in Table~\ref{table:slnn}. As before, the table reports the average runtime, the average number of subproblems, and the average number of Hessian-vector products for HNCG, ANCG, and ACRN. The results show a clear pattern across all tested dimensions: ANCG is consistently much faster than both HNCG and ACRN. In addition, ANCG uses the  fewest Hessian-vector products.

\begin{table}[htbp]
\centering
\resizebox{\textwidth}{!}{
\begin{tabular}{ccc||ccc||ccc||ccc}
\hline
\multicolumn{3}{c||}{Dimension} 
& \multicolumn{3}{c||}{Runtime (seconds)} 
& \multicolumn{3}{c||}{Total subproblems} 
& \multicolumn{3}{c}{Hessian-vector products} \\
$n$ & $m$ & $p$ 
& HNCG & ANCG & ACRN 
& HNCG & ANCG & ACRN 
& HNCG & ANCG & ACRN \\ 
\hline
100  & 20  & 2.25 & 0.02 & 0.01 & 0.31 & 16.2 & 17.0 & 6.2  & 528.6  & 346.6  & 198021.6 \\
100  & 20  & 2.50 & 0.00 & 0.00 & 0.38 & 18.4 & 18.3 & 8.9  & 611.3  & 397.2  & 241037.9 \\
100  & 20  & 2.75 & 0.00 & 0.00 & 0.42 & 19.4 & 19.6 & 10.6 & 678.4  & 431.6  & 272028.1 \\
100  & 20  & 3.00 & 0.01 & 0.00 & 0.50 & 22.3 & 21.1 & 13.3 & 810.8  & 469.7  & 316404.2 \\
500  & 100 & 2.25 & 0.19 & 0.05 & 5.76 & 36.9 & 21.7 & 12.7 & 6385.5 & 1154.0 & 305064.4 \\
500  & 100 & 2.50 & 0.24 & 0.06 & 6.36 & 40.9 & 24.2 & 13.8 & 10044.5& 1470.4 & 332043.9 \\
500  & 100 & 2.75 & 0.33 & 0.07 & 7.29 & 41.4 & 26.2 & 16.5 & 13660.2& 1830.5 & 379043.2 \\
500  & 100 & 3.00 & 0.43 & 0.08 & 7.94 & 44.9 & 28.5 & 18.4 & 18813.1& 2180.7 & 419560.4 \\
1000 & 200 & 2.25 & 1.53 & 0.31 & 27.09 & 42.8 & 23.4 & 15.7 & 16185.7& 1566.9 & 351082.2 \\
1000 & 200 & 2.50 & 2.60 & 0.38 & 30.94 & 45.4 & 25.6 & 18.3 & 29214.5& 2091.2 & 401068.5 \\
1000 & 200 & 2.75 & 4.40 & 0.42 & 36.02 & 48.6 & 27.3 & 21.3 & 50089.2& 2632.8 & 457057.8 \\
1000 & 200 & 3.00 & 7.05 & 0.51 & 38.82 & 50.3 & 30.1 & 23.0 & 83362.3& 3450.8 & 492060.8 \\
\hline
\end{tabular}
}
\caption{Comparison of HNCG, ANCG, and ACRN across different dimensions for the single-layer neural network problem.}
\label{table:slnn}
\end{table}

The two experiments indicate that ANCG solves these problems effectively, achieving the shortest runtime and using the fewest Hessian-vector products. Moreover, both HNCG and ANCG are consistently faster and require fewer Hessian-vector products than ACRN, suggesting a computational advantage of quadratic regularization over higher-order (cubic) regularization schemes that typically require solving more complex subproblems. Furthermore, relative to HNCG, ANCG consistently solves fewer total subproblems (i.e., damped Newton systems), often requiring less than half as many. This reduction highlights the practical benefit of eliminating damping parameter line search from the algorithmic structure.

\section{Conclusion}\label{sec:conclusion}

In this paper, we study Newton-CG methods for finding an  \(\epsilon\)-stationary point of a nonconvex function $f$ whose Hessian is H\"older continuous with modulus $H_f>0$ and exponent \(\nu\in(0,1]\). We identify two limitations in existing methods: non-adaptive regularization and a line-search based damping parameter tuning procedure. The former can lead to inefficiency in the early stage and preclude local superlinear convergence, while the latter may make the algorithms computationally expensive. To circumvent these limitations, we propose two Newton-CG algorithms, depending on the availability of $\nu$, that adaptively regularize the Newton system, 
thereby eliminating the need for line search of the damping parameters. To the best of our knowledge, this work proposes the first Newton-CG method that attains the best-known iteration complexity ${\mathcal O}(H_f^{1/(1+\nu)}\epsilon^{-(2+\nu)/(1+\nu)})$ for nonconvex problems with H\"older-continuous Hessians, while simultaneously enjoying local superlinear convergence. Numerical experiments further validate the practical advantages of our method. 

\appendix
\section{Appendix}\label{apdx}
\renewcommand{\theHsection}{appendix.\Alph{section}}
\subsection{Supporting lemmas}

We start with a technical lemma that will be applied in the analysis of Theorem \ref{thm:1st-cmplx-ancg} and \ref{thm:glb-cplx-2} to remove the $\ln(1/\epsilon)$ factor in the iteration complexities.  
\begin{lemma}\label{lem: sum log upbd}
    For any nonempty set $\mathcal I\subset \mathbb N$, and $p,q,c,M>0$, let $\{z_i\}_{i\in \mathcal I}$ be a positive sequence s.t.  $\sum_{i\in \mathcal I}z_i^p\leq M$. Then it holds that $\sum_{i\in \mathcal I}\ln \frac{z_i}{c}\leq \frac{\max\{M/p,M/q\}}{ec^q}$, regardless of the cardinality of $\mathcal{I}$.
\end{lemma}
\vspace{-12mm}
\begin{proof}
First, applying the Jensen's inequality to $\ln(\cdot)$ function, we obtain that 
\begin{equation*} \begin{aligned}
    \sum_{i\in\mathcal I}\ln \frac{z_i}{c} & = \frac{|\mathcal I|}{p}\frac{\sum_{i\in\mathcal I} \ln (z_i^p)}{|\mathcal I|} +\frac{|\mathcal I|}{q}\ln\left(\frac{1}{c^q}\right)  \leq \frac{|\mathcal I|}{p}\ln\left(\frac{\sum_{i\in\mathcal I} z_i^p}{|\mathcal I|}\right)+\frac{|\mathcal I|}{q}\ln\left(\frac{1}{c^q}\right)\\
    & \leq \frac{|\mathcal I|}{p}\ln\left(\frac{M}{|\mathcal I|}\right)+\frac{|\mathcal I|}{q}\ln\left(\frac{1}{c^q}\right) \leq\max\left\{\frac{1}{p},\frac{1}{q}\right\}|\mathcal I|\ln\left(\frac{M}{c^q|\mathcal I|}\right).
\end{aligned} \end{equation*}
Note that for any $a,b>0$, we have $b\ln(a/b)\leq a/e$, where $e\approx2.718$ is the base of natural logarithm. Applying this fact to the above bound (with $a = Mc^{-q}$ and $b=|\mathcal{I}|$) proves the lemma.
\end{proof}

\subsection{Proof of Section \ref{sec:ncg-hdr}}\label{ssec: pf-sec3}

\subsubsection{Proof of Lemma \ref{lem:trial-xk-ccg}.}  \emph{Proof. } Because the line-search steps guarantee Algorithm \ref{alg:NCG-pd} to be a descent method, we have $\{x^k\}_{k\geq0}\subseteq\mathscr{L}_f(x^0)$. Then fix any $k\geq0$, to show $x^k+\alpha d^k\in\mathscr{L}_f(x^0;r_d)$ for all $\alpha\in[0,1]$, it suffices to show $\|d^k\|\le r_d$, which has two possibilities based on the d$\_$type outputs. 

\noindent\textbf{Case 1.} d$\_$type=SOL. In this case, applying the second inequality of Lemma \ref{lem:ppt-cg}(i) with $g = \nabla f(x^k)$ and $\sigma = \varepsilon_k=(\gamma_k\|\nabla f(x^k)\|^\nu)^{1/(1+\nu)}$, we obtain 
\begin{equation}
\label{eqn:SOL-b-known-nu}
    \|d^k\|\le 1.1(\|\nabla f(x^k)\|/\gamma_k)^{1/(1+\nu)}.
\end{equation}
Together with the fact that $\gamma_k\ge \gamma_0\ge1$ and \eqref{def:some-qtt}, we prove  $\|d^k\|\le 1.1\max\{U_g,1\}\le r_d$.\vspace{0.15cm}

\noindent\textbf{Case 2.} d$\_$type=NC. Line 5 of Algorithm \ref{alg:NCG-pd} and  \eqref{def:some-qtt} yield $\|d^k\|= (d^k)^{\rm T}\nabla^2 f(x^k) d^k/\|d^k\|^2 \le U_H\leq r_d$. 

\noindent Combining these two cases, we complete the proof of this lemma.  \hfill$\Box$

\subsubsection{Proof of Lemma \ref{lem:desc-sol-lip}.}
\emph{Proof. }  Throughout this proof, we will frequently use the shorthand $\varepsilon_k=(\gamma_k\|\nabla f(x^k)\|^\nu)^{1/(1+\nu)}$ defined in Algorithm \ref{alg:NCG-pd} to simplify the notation. As the algorithm does not terminate, we have $\|\nabla f(x^k)\|\neq0$ and hence $\varepsilon_k>0$. Because d\_type = SOL, we can apply the fourth inequality of Lemma \ref{lem:ppt-cg}(i) to yield $d^k\neq0$. Moreover, for $k\in\mathbb{K}_{\epsilon,1}$, we also have $\|\nabla f(x^{k+1})\| \ge \|\nabla f(x^k)\| / 2$ and $\sigma_k \le 2\gamma_k$. With the above information, we can start the proof. \vspace{0.2cm}

\noindent\textbf{Statement (i). } If $\alpha_k=1$, the statement clearly holds. If $\alpha_k<1$, then we know $j_k\geq 1$. In this case, using the definition of $\sigma_k$ (Algorithm \ref{alg:NCG-pd} Line 9) for $j\in\{0,j_k-1\}$ that violates \eqref{ls-sol-stepsize}, we have
\vspace{-2mm}
\begin{equation*}
\begin{aligned}
&-\eta \varepsilon_k\theta^j\|d^k\|^2 \le f(x^k + \theta^j d^k) - f(x^k)\\
&\overset{(i)}{\le} \theta^j\nabla f(x^k)^{\rm T}d^k + \frac{\theta^{2j}}{2}(d^k)^{\rm T}\nabla^2 f(x^k) d^k + \mathcal{R}_0(x^k + \theta^j d^k,x^k)\\
&\overset{(ii)}{=}-\theta^j\Big(1-\frac{\theta^j}{2}\Big)(d^k)^{\rm T}\Big(\nabla^2 f(x^k)+2\varepsilon_kI\Big)d^k - \theta^{2j}\varepsilon_k\|d^k\|^2 + \frac{\mathcal{H}_0(x^k + \theta^j d^k,x^k)\|\theta^jd^k\|^{2+\nu}}{2}\\
&\overset{(iii)}{\le} -\theta^j\varepsilon_k\|d^k\|^2 + \frac{\sigma_k\theta^{(2+\nu)j}\|d^k\|^{2+\nu}}{2},
\end{aligned}
\end{equation*}  
where (i) is by the triangle inequality and the definition \eqref{def:R0}, (ii) is by the third relation of Lemma \ref{lem:ppt-cg}(i) and the definition \eqref{defn:estimators}, and (iii) is by the first inequality of Lemma \ref{lem:ppt-cg}(i), and the definition of $\sigma_k$ in Algorithm \ref{alg:NCG-pd} Line 9, which is effective for $j=0$ or $j=j_k-1$. As $d^k\neq0$, dividing both sides of the above inequality by $\sigma_k\theta^{j}\|d^k\|^{2+\nu}/2$ yields
\begin{equation}
\label{eqn:theta-lower-known-nu-Sol}
\theta^{(1+\nu)j} \ge \frac{2(1-\eta)\varepsilon_k}{\sigma_k\|d^k\|^\nu}, \qquad j\in\{0,j_k-1\}.  
\end{equation}
Then setting $j=j_k-1$ in the above inequality gives 
\begin{equation}\begin{aligned}\label{ineq:lwbd-ak-sol}
\alpha_k = \theta^{j_k} \ge \bigg(\frac{2(1-\eta)\varepsilon_k}{\sigma_k\|d^k\|^\nu}\bigg)^{1/(1+\nu)}\theta \ge \bigg(\frac{2(1-\eta)\gamma_k}{1.1^\nu \sigma_k}\bigg)^{1/(1+\nu)}\theta\ge \bigg(\frac{2(1-\eta)\gamma_k}{1.1^\nu H_f}\bigg)^{1/(1+\nu)}\theta,    
\end{aligned}\end{equation}
where the second inequality is by the definition $\varepsilon_k=(\gamma_k\|\nabla f(x^k)\|^\nu)^{1/(1+\nu)}$ and \eqref{eqn:SOL-b-known-nu}, and the last inequality is because $\sigma_k$ always underestimates $H_f$. \vspace{0.2cm}

\noindent\textbf{Statement (ii).} Now we prove this statement by considering two separate cases below.

\noindent\textbf{Case 1. } $\alpha_k=1$. In this case, we have $x^{k+1} =x^k + d^k$ and it holds that 
\begin{equation*} 
\begin{aligned}
\|\nabla f(x^{k+1})\| & \le \mathcal{R}_1(x^k+d^k,x^k) + \|(\nabla^2 f(x^k) + 2\varepsilon_kI)d^k + \nabla f(x^k)\| + 2\varepsilon_k\|d^k\|\\
&\le {\sigma_k}\|d^k\|^{1+\nu} + \frac{5}{2}\varepsilon_k\|d^k\|,
\end{aligned}\end{equation*}
where the first line is by the triangle inequality, and the second line is by the definition of $\sigma_k$ (Algorithm \ref{alg:NCG-pd} Line 9), the definition \eqref{defn:estimators}, the fourth relation of Lemma \ref{lem:ppt-cg}(i), and {the definition of $\zeta_k$ in Line 3 of Algorithm \ref{alg:NCG-pd}}. As a consequence, at least one of $\sigma_k\|d^k\|^{1+\nu}\ge\frac{\|\nabla f(x^{k+1})\|}{2}$ and $2.5\varepsilon_k\|d^k\|\ge\frac{\|\nabla f(x^{k+1})\|}{2}$ will hold. Combined with the definition of $\mathbb{K}_{\epsilon,1}$ that $\sigma_k\le 2\gamma_k$, and $\|\nabla f(x^{k+1})\|\ge\|\nabla f(x^k)\|/2$, we know $\|d^k\| \ge \gamma_k^{-\frac{1}{1+\nu}}\|\nabla f(x^k)\|^{\frac{1}{1+\nu}}/10$. By this lower bound of $\|d^k\|$ and \eqref{ls-sol-stepsize}, we complete the proof of \eqref{lwbd:f-descent} when $\alpha_k=1$.\vspace{0.1cm}

\noindent\textbf{Case 2. } $\alpha_k<1$. In this case $j_k\geq1$. Choosing $j=0$ in  \eqref{eqn:theta-lower-known-nu-Sol}  lower bounds $\|d^k\|^\nu\ge2(1-\eta)\varepsilon_k/\sigma_k$. Together with the first inequality of \eqref{ineq:lwbd-ak-sol}, $\sigma_k\le2\gamma_k$, and \eqref{ls-sol-stepsize}, we obtain that 
\begin{equation*}\begin{aligned}
& f(x^k) - f(x^{k+1}) \ge \eta \varepsilon_k\alpha_k\|d^k\|^2\ge \eta (1-\eta)^{\frac{2}{\nu}}\theta \gamma_k^{-\frac{1}{1+\nu}}\|\nabla f(x^k)\|^{\frac{2+\nu}{1+\nu}},
\end{aligned}\end{equation*}
which implies \eqref{lwbd:f-descent} in this case. 

\noindent\textbf{Next gradient bound. } By the inequalities \eqref{ineq:desc-hd}, \eqref{eqn:SOL-b-known-nu}, the fact that $\alpha_k\leq1$, $\gamma_k\geq1$, and the definition of $\zeta_k$ in Algorithm \ref{alg:NCG-pd} Line 3, and the first and last inequalities of Lemma \ref{lem:ppt-cg}(i), one has that
\begin{equation*}\begin{aligned}
&\,\,\|\nabla f(x^{k+1})\| = \|\nabla f(x^k +\alpha_k d^k)\|\\
& \le \mathcal{R}_1(x^k\!+\!\alpha_kd^k\!,x^k) \!+\! \alpha_k\|(\nabla^2\! f(x^k) \!+\! 2\varepsilon_kI)d^k \!+\! \nabla f(x^k)\| + 2\alpha_k\varepsilon_k\|d^k\| \!+\! (1\!-\!\alpha_k)\|\nabla f(x^k)\|\\
&\le {H_f}\|d^k\|^{1+\nu} + \frac{5\alpha_k\varepsilon_k}{2}\|d^k\| + (1\!-\!\alpha_k)\|\nabla f(x^k)\|\\
&\leq (1.1)^{1+\nu}\frac{H_f}{\gamma_k}\|\nabla f(x^k)\| + (1+1.75\alpha_k)\|\nabla f(x^k)\| \leq  (2H_f+5) \|\nabla f(x^k)\|. 
\end{aligned}\end{equation*} 
Hence, we complete the proof of this lemma.  \hfill$\Box$

\subsubsection{Proof of Lemma \ref{lem:desc-nc-lip}.}\emph{Proof. } As the algorithm does not terminate, we know $\|\nabla f(x^k)\|\neq0$. 
By Lemma \ref{lem:ppt-cg}(ii), we also confirm that $d^k\neq0$ in this case. We should also note that the search direction $d^k$ in this lemma is not the raw output of \texttt{CappedCG}. When d$\_$type=NC, it undergoes an additional re-normalization step (Algorithm \ref{alg:NCG-pd} Line 5) such that 
\vspace{-0.2cm}
\begin{equation}\label{NCG-nu-2-ppty-alg1}
\nabla f(x^k)^{\rm T} d^k\le 0\qquad\mbox{and}\qquad \frac{(d^k)^{\rm T}\nabla^2 f(x^k) d^k}{\|d^k\|^2} = -\|d^k\| < -\varepsilon_k,
\end{equation}
where the second inequality is by Lemma \ref{lem:ppt-cg}(ii). Since $k\in\mathbb{K}_{\epsilon,1}$, we have that $\|\nabla f(x^{k+1})\| \ge \|\nabla f(x^k)\| / 2$ and $\sigma_k \le 2\gamma_k$. With these in mind, let us prove the two statements one by one.   \vspace{0.1cm}

\noindent\textbf{Statement (i). }  If \eqref{ls-nc-stepsize} holds $j=0$, then $\alpha_k=1$, which clearly implies this statement. Now let us consider the case $j_k\geq1$ and $\alpha_k = \theta^{j_k}<1$. Because \eqref{ls-nc-stepsize} fails for $j=j_k-1$, one has
\vspace{-0.2cm}
\begin{equation*}\begin{aligned}
-\frac{\eta}{2}\theta^{2j}\|&d^k\|^3  \!\le f(x^k \!+ \theta^j d^k) \!-\! f(x^k) \!\le \theta^j \nabla\! f(x^k)^{\!\rm T}\! d^k \!+\! \frac{\theta^{2j}}{2} (d^k)^{\!\rm T} \nabla^2\! f(x^k) d^k \!\!+\! \mathcal{R}_0(x^k \!+\! \theta^j d^k\!,x^k) \\
&{\leq}  - \frac{\theta^{2j}}{2}\|d^k\|^3 +  \frac{\mathcal{H}_0(x^k + \theta^j d^k,x^k)\|\theta^jd^k\|^{2+\nu}}{2}= - \frac{\theta^{2j}}{2}\|d^k\|^3 + \frac{\sigma_k\theta^{(2+\nu)j}\|d^k\|^{2+\nu}}{2},
\end{aligned}\end{equation*}
where the second line is due to \eqref{NCG-nu-2-ppty-alg1} and the definition \eqref{defn:estimators}. 
Next, dividing both sides by $\theta^{2j}\sigma_k\|d^k\|^{2+\nu}/2$ yields that $\theta^{\nu j} \ge (1-\eta)\|d^k\|^{1-\nu}/\sigma_k.$ Combining this with $\alpha_k = \theta^{j_k}$, $\sigma_k\le H_f$, and $\|d^k\|\ge\varepsilon_k$, and setting $j=j_k-1$ in the above inequality gives  
{\setlength{\abovedisplayskip}{1pt}
\begin{align}\label{ineq:lwbd-ak-vio-nc}
\alpha_k=\theta^{j_k}\ge \theta(1-\eta)^{\frac1\nu}\|d^k\|^{\frac{1-\nu}{\nu}}/\sigma_k^{\frac{1}{\nu}} \ge \theta(1-\eta)^{\frac{1}{\nu}}\gamma_k^{\frac{1-\nu}{\nu(1+\nu)}}\|\nabla f(x^k)\|^{\frac{1-\nu}{1+\nu}}/H_f^{\frac{1}{\nu}},    
\end{align}}
which along with $\gamma_k\ge1$ yields statement (i) as desired.  \vspace{0.1cm} 

\noindent\textbf{Statement (ii). } If $\alpha_k=1$, it follows from \eqref{ls-nc-stepsize}, $\|d^k\|\geq\varepsilon_k$, and $\gamma_k\ge 1$ that
\begin{equation*}
f(x^k)-f(x^{k+1})\ge\frac{\eta}{2}\varepsilon_k^3=\frac{\eta}{2}(\gamma_k\|\nabla f(x^k)\|^\nu)^{\frac{3}{1+\nu}}\ge\frac{\eta}{2} \gamma_k^{-\frac{1}{1+\nu}}\min\left\{\|\nabla f(x^k)\|^{\frac{2+\nu}{1+\nu}},1\right\},    
\end{equation*}
which implies \eqref{ineq:nc-suff-desc}. If $\alpha_k<1$, $\|d^k\|\geq\varepsilon_k$, the first inequality of \eqref{ineq:lwbd-ak-vio-nc}, and \eqref{ls-nc-stepsize} imply
\begin{equation*}
f(x^k) - f(x^{k+1}) \ge \frac{\eta}{2}\alpha_k^{2}\varepsilon_k^3 \ge \frac{\eta}{2}\theta^2\bigg(\frac{(1-\eta)\gamma_k}{\sigma_k}\bigg)^{\frac{2}{\nu}}\gamma_k^{-\frac{1}{1+\nu}}\|\nabla f(x^k)\|^{\frac{2+\nu}{1+\nu}},     
\end{equation*}
combined with $\gamma_k/\sigma_k\geq1/2$, we complete the proof of statement (ii). 

\noindent\textbf{Next gradient bound. } Given $\|d^k\|\leq M =\min\{U_H^\nu,\theta(1-\eta)^{\frac1\nu}U_H/H_f^{\frac{1}{\nu}}\}$, let us bound the next gradient. 
When $\alpha_k=1$, by $\|d^k\|\le U_H^{\nu}$, we have $U_H\|d^k\|\geq\|d^k\|^{\frac{1+\nu}{\nu}}.$
When $\alpha_k<1$, by the first inequality of \eqref{ineq:lwbd-ak-vio-nc}, $\sigma_k\leq H_f$, and $\|d^k\|\leq\theta(1-\eta)^{\frac1\nu}U_H/H_f^{\frac{1}{\nu}}$, we have 
$\alpha_kU_H\|d^k\|\geq \|d^k\|^{\frac{1+\nu}{\nu}}.$
Overall, using Assumptions \ref{asp:basic} and \ref{asp:hd}, we obtain that $\nabla f$ is Lipschitz continuous, which further implies that 
\begin{equation*}
    \|\nabla f(x^{k+1})\|\leq \|\nabla f(x^k)\|+\alpha_kU_H\|d^k\|\overset{(i)}{\le} \|d^k\|^{\frac{1+\nu}{\nu}}+\alpha_k U_H\|d^k\|\leq2\alpha_kU_H\|d^k\|.
\end{equation*}
where (i) is because $\|d^k\|\geq\varepsilon_k\geq\|\nabla f(x_k)\|^{\frac{\nu}{1+\nu}}.$ Hence we complete the proof.   \hfill$\Box$

\subsubsection{Proof of Theorem~\ref{thm:1st-cmplx-ancg}.}\label{appdx:Thm-1}  \emph{Proof. }  Recall that $\mathbb{K}_\epsilon:=\{k:\|\nabla f(x^t)\|>\epsilon, \forall t\leq k\}$ contains all iterations before termination. And by definition, $|\mathbb{K}_{\epsilon,i}|$, $i=1,2,3$, form a partition of $\mathbb{K}_{\epsilon}$. \vspace{0.1cm}

\noindent\textbf{Part 1. } Bounding $|\mathbb{K}_{\epsilon,1}|$. Because  $\gamma_k\le H_f$ for all $k\in\mathbb{K}_\epsilon$, combining Lemmas \ref{lem:desc-sol-lip} and \ref{lem:desc-nc-lip}, we know that each iteration $k\in\mathbb{K}_{\epsilon,1}$ results in either a constant descent in the  objective value $f(x^k)-f(x^{k+1})\geq c_{\rm nc,\nu}H_f^{-{1}/{(1+\nu)}}$, or a sufficient descent of 
\vspace{-0.2cm}
$$f(x^k)-f(x^{k+1})\geq \min\{c_{\mathrm{sol,\nu}},c_{\mathrm{nc,\nu}}\}\gamma_k^{-\frac{1}{1+\nu}}\|\nabla f(x^k)\|^{\frac{2+\nu}{1+\nu}} \geq \min\{c_{\mathrm{sol,\nu}},c_{\mathrm{nc,\nu}}\}H_f^{-\frac{1}{1+\nu}}\epsilon^{\frac{2+\nu}{1+\nu}}.$$
Because Algorithm \ref{alg:NCG-pd} is a descent method, and recalling that $\Delta_f=f(x^0)-f_{\rm low}$, we have  
\begin{equation}\label{ineq:upbd-k1-alg1}
|\mathbb{K}_{\epsilon,1}|\le \frac{\Delta_fH_f^{\frac{1}{1+\nu}}\epsilon^{-\frac{2+\nu}{1+\nu}}}{\min\{c_{\mathrm{sol,\nu}},c_{\mathrm{nc},\nu}\}} +\frac{\Delta_fH_f^{\frac{1}{1+\nu}}}{c_{\rm nc,\nu}}= \mathcal{O}\left(\Delta_fH_f^{\frac{1}{1+\nu}}\epsilon^{-\frac{2+\nu}{1+\nu}}\right) .\end{equation}

\noindent\textbf{Part 2. } Bounding $|\mathbb{K}_{\epsilon,2}|$.  According to Line 10 of Algorithm \ref{alg:NCG-pd}, the sequence $\{\gamma_k\}_{k\geq0}$ is nondecreasing and is always upper bounded by $H_f$. For every $k\in\mathbb{K}_{\epsilon,2}$, its definition requires $\sigma_k\geq2\gamma_k$, and Line 10 of Algorithm \ref{alg:NCG-pd} states that next $\gamma_{k+1} = \max\{\gamma_k,\sigma_k\}\geq2\gamma_k$ will at least double the size. Therefore, we can bound $|\mathbb{K}_{\epsilon,2}|$ by 
\begin{equation*}
|\mathbb{K}_{\epsilon,2}|\le\lceil\ln H_f/\ln 2\rceil+1 = \mathcal{O}(1).
\end{equation*} 

\noindent\textbf{Part 3. }  Bounding $|\mathbb{K}_{\epsilon,3}|$. The upper bound of this subset is a bit complicated. For all $k\in\mathbb{K}_{\epsilon,3}$, the next gradient halves: $\|\nabla f(x^{k+1})\|\leq \|\nabla f(x^{k})\|/2$. Note that $\mathbb{K}_{\epsilon,1}$, $\mathbb{K}_{\epsilon,2}$ are finite and $\mathbb{K}_{\epsilon}=\mathbb{K}_{\epsilon,1}\cup \mathbb{K}_{\epsilon,2}\cup \mathbb{K}_{\epsilon,3}$, one can see that $\mathbb{K}_{\epsilon,3}$ can be partitioned into $|\mathbb{K}_{\epsilon,1}|+|\mathbb{K}_{\epsilon,2}|+1$ disjoint subsets of \emph{consecutive} nonnegative integers. Now, for those subsets that follows an iteration index $k\in \mathbb{K}_{\epsilon,1}$, we denote them by $\mathbb{K}_{\epsilon,3}^{i}$ with $i\in\mathcal{I}_1:=\{1,\cdots,\ell_1\}$. For those subsets that follows an iteration index $k\in\mathbb{K}_{\epsilon,2}$ or in case the subset contains $0$, we denote then by $\mathbb{K}_{\epsilon,3}^{i}$ with $i\in\mathcal{I}_2:=\{\ell_1+1,\cdots,\ell_1+\ell_2\}$. Then clearly,  we have $\ell_1\leq |\mathbb{K}_{\epsilon,1}|$ and $\ell_2\leq |\mathbb{K}_{\epsilon,2}|+1$.

Because $\epsilon\leq\|\nabla f(x^k)\|\leq U_g$ for all $k\in\mathbb{K}_{\epsilon}$, each subset has at most $\lceil\ln (U_g/\epsilon)/\ln 2\rceil+1$ iterations. Consequently, for those subsets following an $\mathbb{K}_{\epsilon,2}$ iteration, we have 
\begin{equation*}
\sum_{i\in\mathcal{I}_2} |\mathbb{K}_{\epsilon,3}^{i}|\leq \left(|\mathbb{K}_{\epsilon,2}|+1\right)\left(\left\lceil\frac{\ln (U_g/\epsilon)}{\ln 2}\right\rceil+1\right) \leq  \mathcal{O}\left(\ln\frac{1}{\epsilon}\right).
\end{equation*}

Next, for those subsets that follow a $\mathbb{K}_{\epsilon,1}$ iteration (in $\mathcal{I}_1$), directly applying the above bound would result in an undesirable $\mathcal{O}\big(\epsilon^{-\frac{2+\nu}{1+\nu}}\ln\frac{1}{\epsilon}\big)$ complexity. To remove the extra logarithmic factor, let us further partition these  subsets into $\mathcal{I}_1=\mathcal{I}_{\rm sol}\cup\mathcal{I}_{\rm nc}$ by whether the subset follows an iteration $k\in \mathbb{K}_{\epsilon,1}$ with d\_type=SOL or d\_type = NC, respectively. For notational simplicity, let us suppose the initial iteration of each $\mathbb{K}_{\epsilon,3}^i$ has gradient norm $\delta_i$. The last iteration of $\mathbb{K}_{\epsilon,1}$ prior to each $\mathbb{K}^i_{\epsilon,3}$ has gradient norm $\bar{\delta}_i$, search direction norm $\hat{\delta}_i$ and line-search step length $\hat\alpha_i$.

To upper bound $\sum_{i\in\mathcal{I}_{\rm sol}}|\mathbb{K}^i_{\epsilon,3}|$, by the relationship $\delta_i\leq (2H_f+5)\bar{\delta}_i$ from Lemma \ref{lem:desc-sol-lip}, we have 
\begin{equation*}
\begin{aligned}
\sum_{i\in\mathcal{I}_{\rm sol}}|\mathbb{K}_{\epsilon,3}^{i}&| \leq   \sum_{i\in\mathcal{I}_{\rm sol}} \left(\left\lceil\frac{\ln (\delta_i/\epsilon)}{\ln 2}\right\rceil+1\right)    
{\le}  2(1+\ln(2H_f+5))|\mathcal{I}_{\rm sol}| +2\sum_{i\in\mathcal{I}_{\rm sol}} \ln (\bar\delta_i/\epsilon) \\ 
\leq& 2\left(1+\ln(2H_f+5)\right)|\mathbb{K}_{\epsilon,1}| + \frac{4\Delta_fH_f^{\frac{1}{1+\nu}}\epsilon^{-\frac{2+\nu}{1+\nu}}}{3ec_{\mathrm{sol},\nu}} = \mathcal{O}\left(\Delta_fH_f^{\frac{1}{1+\nu}}\epsilon^{-\frac{2+\nu}{1+\nu}}\right),  
\end{aligned}\end{equation*}
where the third inequality uses $|\mathcal{I}_{\rm sol}|\leq |\mathbb{K}_{\epsilon,1}|$, $\sum_{i\in\mathcal{I}_{\rm sol}}\bar\delta_i^{\frac{2+\nu}{1+\nu}}\le H_f^{\frac{1}{1+\nu}}\Delta_f/c_{{\rm sol}, \nu}$ (due to \eqref{lwbd:f-descent} and $\gamma_k\le H_f$), and Lemma \ref{lem: sum log upbd} with $(z_i,p,q,c)=(\bar\delta_i,\frac{2+\nu}{1+\nu},\frac{2+\nu}{1+\nu},\epsilon)$.

For $\mathcal{I}_{\rm nc}$, we divide it into $\mathcal{I}_{\rm nc}^{\leq}:=\{i\in\mathcal{I}_{\rm nc}:\hat\delta_i\leq M\}$ with $M$ defined in Lemma \ref{lem:desc-nc-lip}, and $\mathcal{I}_{\rm nc}^{>}:=\mathcal{I}_{\rm nc}\setminus\mathcal{I}_{\rm nc}^{\leq}$. For each $i\in\mathcal{I}_{\rm nc}^>$, by \eqref{ls-nc-stepsize} we know the last iteration in $\mathbb{K}_{\epsilon,1}$ will yield a descent of at least 
$$\frac{\eta}{2}\hat\alpha_i^2\hat\delta_i^3 \ge \frac{\eta}{2}\hat\delta_i^3\cdot\min\Big\{1, \theta^2(1-\eta)^{\frac2\nu}\hat\delta_i^{\frac{2-2\nu}{\nu}}H_f^{-\frac{2}{\nu}}\Big\} \geq  \frac{\eta}{2}\min\Big\{M^3,\theta^2(1-\eta)^{\frac2\nu} M^{\frac{2+\nu}{\nu}}H_f^{-\frac{2}{\nu}} \Big\}=:C,$$
where the first inequality is due to the first inequality of \eqref{ineq:lwbd-ak-vio-nc} and the fact that $\sigma_k\leq H_f$. As the total descent is controlled by $\Delta_f$, then we easily upper bound $|\mathcal{I}_{\rm nc}^{>}|$ by ${\Delta_f}/{C} = \mathcal{O}(1)$ and obtain
\begin{equation}
    \label{eqn:bound-Inc>}
    \sum_{i\in\mathcal{I}_{\rm nc}^>}|\mathbb{K}^i_{\epsilon,3}|\leq |\mathcal{I}_{\rm nc}^{>}|\cdot\bigg(\bigg\lceil\frac{\ln (U_g/\epsilon)}{\ln 2}\bigg\rceil+1\bigg) \leq \mathcal{O}\left(\ln\frac{1}{\epsilon}\right).
\end{equation}
For $i\in \mathcal{I}_{\rm nc}^\leq$, again, \eqref{ls-nc-stepsize} guarantees the last iteration, indexed by $k_i$, in $\mathbb{K}_{\epsilon,1}$ to generate a descent of at least 
$\frac{\eta}{2}\hat\alpha_i^2\hat\delta_i^3 \geq \frac{\eta}{2}\hat\alpha_i^2\hat\delta_i^2\epsilon^{\frac{\nu}{1+\nu}}$, where we single out one $\hat{\delta}_i$ and lower bound it by \eqref{NCG-nu-2-ppty-alg1}, under the fact that $\hat{\delta}_i\geq \varepsilon_{k_i} = (\gamma_{k_i}\|\nabla f(x^{k_i})\|^\nu)^{1/(1+\nu)}\geq\epsilon^{\frac{\nu}{1+\nu}}$. Again, using $\Delta_f$ to control the total descent gives 
\begin{equation}
    \label{eqn:NC-nu-upperDescent}
    \sum_{i\in\mathcal{I}_{\rm nc}^{\leq}}\hat\alpha_i^2\hat\delta_i^2\leq\frac{2\Delta_f}{\eta\cdot \epsilon^{\frac{\nu}{1+\nu}}},
\end{equation}
which shall be used later as a summation upper bound in Lemma \ref{lem: sum log upbd}. With this in mind, and use the inequality $\delta_i\leq 2\hat{\alpha}_iU_H\hat{\delta}_i$ provided by Lemma \ref{lem:desc-nc-lip}, we have
\begin{equation}
\label{eqn:bound-Inc<=}
\begin{aligned}
\sum_{i\in\mathcal{I}^\leq_{\rm nc}}|\mathbb{K}_{\epsilon,3}^{i}| \leq & \sum_{i\in\mathcal{I}^\leq_{\rm nc}} \left(\left\lceil\frac{\ln (\delta_i/\epsilon)}{\ln 2}\right\rceil+1\right) 
\le   2\big(1+\ln(2U_H)\big)|\mathcal{I}^\leq_{\rm nc}|+2\sum_{i\in\mathcal{I}^\leq_{\rm nc}} \ln \left(\frac{\hat\alpha_i\hat\delta_i}{\epsilon}\right)\\ 
{\le} & 2\big(1+\ln(2U_H)\big)|\mathbb{K}_{\epsilon,1}|+ \frac{4\Delta_f}{e\eta\epsilon^{(1+2\nu)/(1+\nu)}} =  \mathcal{O}\left(\Delta_fH_f^{\frac{1}{1+\nu}}\epsilon^{-\frac{2+\nu}{1+\nu}}\right),\\
\end{aligned}\end{equation}
where the third relation is by $|\mathcal{I}^\leq_{\rm nc}|\leq|\mathbb{K}_{\epsilon,1}|$, the summation upper bound \eqref{eqn:NC-nu-upperDescent}, and Lemma \ref{lem: sum log upbd} with $(z_i,p,q,c)=(\hat\alpha_i\hat\delta_i,2,1,\epsilon)$. Synthesizing the inequalities \eqref{ineq:upbd-k1-alg1}--\eqref{eqn:bound-Inc>} and \eqref{eqn:bound-Inc<=}, we obtain 
$$|\mathbb{K}_{\epsilon}| = |\mathbb{K}_{\epsilon,1}|+|\mathbb{K}_{\epsilon,2}|+\sum_{i\in\mathcal{I}_{2}}|\mathbb{K}_{\epsilon,3}^i|+\sum_{i\in\mathcal{I}_{\rm sol}}|\mathbb{K}_{\epsilon,3}^i|+\sum_{i\in\mathcal{I}_{\rm nc}^>}|\mathbb{K}_{\epsilon,3}^i|+\sum_{i\in\mathcal{I}_{\rm nc}^\leq}|\mathbb{K}_{\epsilon,3}^i|\leq\mathcal{O}\left(\Delta_fH_f^{\frac{1}{1+\nu}}\epsilon^{-\frac{2+\nu}{1+\nu}}\right),$$
hence we complete the proof of this theorem.                      \hfill$\Box$

\subsubsection{Proof of Theorem \ref{thm:loc-conv-1}.} 
\label{appdx:Thm-local-Known-Nu} \emph{Proof. } First, let us define the radius constant $\delta$ in Theorem \ref{thm:loc-conv-1} as
\begin{equation*}
    \delta=\min\bigg\{\frac{1}{2L_gH_f^{1/\nu}},\bigg(\frac{\mu}{2H_f}\bigg)^{\frac1\nu},\frac{1}{4L_g}\bigg(\frac{\mu}{2H_f+4H_f^\frac{1}{1+\nu}L_g^{\frac{\nu}{1+\nu}}+2L_g^{\frac{1+2\nu}{1+\nu}}}\bigg)^{\frac{1+\nu}{\nu}}\bigg\},
\end{equation*} 
where $L_g:=\|\nabla^2 f(x^\ast)\|+1$. By $x^{k_0}\in\mathscr{L}_f(x^0)$, and $\delta\leq \frac{1}{2H_f^{1/\nu}}<\frac{r_d}{2}$, we know  $B_{\delta}(x^\ast)\subseteq\mathscr{L}_f(x^0;r_d)$. Consequently, we can apply the $(H_f,\nu)$-H\"older continuity of $\nabla^2f$ in $B_{\delta}(x^\ast)$ and obtain 
\begin{equation}
    \label{eqn:loc-SC}
    \frac{\mu}{2}I\preceq \nabla^2 f(x)\preceq L_g I,\qquad\forall x\in B_{\delta}(x^\ast),
\end{equation}
where the first half of inequality uses $\delta\leq \big(\frac{\mu}{2H_f}\big)^{1/\nu}$.  With this property, when a point $x^k\in B_{\delta}(x^*)$, the subroutine \texttt{CappedCG} will always output d\_type=SOL because $\nabla^2f(x^k)\succ0$ has no negative curvature directions. Now we would like to claim that the stepsize output by the \texttt{CappedCG} subroutine will always take $\alpha_k=1$ when  $x^k\in B_{\delta}(x^*)$ by verifying the line search condition \eqref{ls-sol-stepsize}:
 \begin{equation*}\begin{aligned}
    f(x^k+d^k)-f(x^k)&\leq \nabla f(x^k)^{\top}d^k+\frac{1}{2}(d^{k})^\top\nabla^2f(x^k)d^k+\frac{H_f}{2}\|d^k\|^{2+\nu}\\
    &\overset{(i)}{=} -2\varepsilon _k\|d^k\|^2-\frac{1}{2}d^{k\top}\nabla^2f(x^k)d^k+\frac{H_f}{2}\|d^k\|^{2+\nu}\\
    &\overset{(ii)}{\leq} -2\varepsilon_k\|d^k\|^2-\frac{\mu}{4}\|d^k\|^2+\frac{\mu}{4}\|d^k\|^2 \leq -\eta\varepsilon_k\|d^k\|^2,
\end{aligned}\end{equation*}
where (i) is by third relation of Lemma \ref{lem:ppt-cg}(i) and (ii) is by \eqref{eqn:loc-SC}, and \eqref{eqn:SOL-b-known-nu} that implies 
 \begin{equation*} 
    \|d^k\|^\nu\le 1.1^\nu(\|\nabla f(x^k)\|)^{\frac{\nu}{1+\nu}}\leq 1.1^\nu(L_g\|x^k-x^*\|)^{\frac{\nu}{1+\nu}}\leq 1.1^\nu(L_g\delta)^{\frac{\nu}{1+\nu}}\leq \frac{\mu}{2H_f}.
\end{equation*}
Now, we confirm that for Algorithm \ref{alg:NCG-pd}, all once an iteration enters the $B_{\delta}(x^*)$ neighborhood, it will always execute the Newton step. It remains to show that all future iterations will remain in this neighborhood and they will converge superlinearly to the locally optimal solution $x^\ast$. Now suppose $x^k\in B_{\delta}(x^*)$ for some $k\geq k_0$, then according to the previous argument, $x^{k+1} = x^k+d^k$, we have
{
\setlength{\abovedisplayskip}{1pt}
\begin{align}
    \|&x^{k+1}\!-x^\ast\|
 \overset{(i)}{\le} \|(\nabla^2\!f(x^k)+2\varepsilon_k I)^{-1}\nabla f(x^k)+x^k\!-x^\ast\|
   +\|d^k\!+(\nabla^2f(x^k)+2\varepsilon_k I)^{-1}\nabla\! f(x^k)\|\notag\\
&\le
   \|(\nabla^2\!f(x^k)+2\varepsilon_k I)^{-1}\|
   \Big(\|\nabla f(x^k)+\nabla^2f(x^k)(x^k-x^\ast)\|+2\varepsilon_k\|x^k-x^\ast\|
      +\zeta_k\|\nabla f(x^k)\|\Big)\notag\\ 
& \overset{(ii)}{\le}    \frac{2}{\mu}\Big(H_f\|x^k-x^\ast\|^{1+\nu}
   +2\gamma_k^\frac{1}{1+\nu}\|\nabla f(x^k)\|^\frac{\nu}{1+\nu}\|x^k-x^\ast\|
   +\|\nabla f(x^k)\|^\frac{\ 1+2\nu}{1+\nu}\Big)\notag\\
&  \overset{(iii)}{\le} \frac{2}{\mu}\Big(H_f\delta^{1+\nu}
   +2H_f^\frac{1}{1+\nu}L_g^\frac{\nu}{1+\nu}\delta^{\frac{1+2\nu}{1+\nu}}
   +(L_g\delta)^\frac{1+2\nu}{1+\nu}\Big) {\leq} \delta,\notag
\end{align} 
}
where (i) is by triangle inequality,  (ii)  is by \eqref{ineq:desc-hd} and the definition of $\varepsilon_k,\zeta_k$ in Line 3 of Algorithm~\ref{alg:NCG-pd},  and (iii) is by \eqref{eqn:loc-SC}. Based on this bound, we can conclude by induction that $\{x^k\}_{k\geq k_0}\subseteq B_{\delta}(x^*)$. Finally, by reusing the step (ii) in the above inequality, we establish the superlinear rate for $k\geq k_0$ 
\begin{equation*}\begin{aligned}
\|\nabla f(x^{k+1})\|
&\le
 \frac{2L_g}{\mu}\Big(H_f\|x^k-x^\ast\|^{1+\nu}
   +2\gamma_k^\frac{1}{1+\nu}\|\nabla f(x^k)\|^\frac{\nu}{1+\nu}\|x^k-x^\ast\|
   +\|\nabla f(x^k)\|^\frac{1+2\nu}{1+\nu}\Big)\\
&\le \frac{2L_g}{\mu}\left(
   \frac{2^{1+\nu}H_f}{\mu^{1+\nu}}\|\nabla f(x^k)\|^{1+\nu}
   +\left(4\mu^{-1}H_f^\frac{1}{1+\nu}+1\right)\|\nabla f(x^k)\|^\frac{1+2\nu}{1+\nu}\right)
\end{aligned}\end{equation*}
That is, $\|\nabla f(x^{k+1})\| = \mathcal{O}(\|\nabla f(x^k)\|^{\frac{1+2\nu}{1+\nu}})$ for $k\geq k_0$, suggesting a superlinear rate of order $\frac{1+2\nu}{1+\nu}$.  \hfill $\Box$

\subsection{Proof of Section \ref{sec:ncg-hd}.}\label{ssec: pf-sec4} 

Before presenting the proofs, we first provide supporting inequalities and lemmas. By direct computation, the following equivalence holds for any $z>0$:
\begin{equation}\label{ineq:tech-gam>}
\gamma\geq\gamma_\nu(z)\quad\Longleftrightarrow\quad\frac{(\gamma z)^{1/2}}{H_f} \ge 2^{1+\nu}\Big(\frac{z}{\gamma}\Big)^{\nu/2}\quad\Longleftrightarrow\quad \frac{(\gamma z)^{(1-\nu)/2}}{H_f} \ge \frac{2^{1+\nu}}{\gamma^{\nu}}.
\end{equation}
Its proof can be found in Lemma 3 of \citet{he2025newton} and is omitted here. The following lemma, used to establish sufficient descent later, shows that when $\gamma_k$ is sufficiently large, either the gradient norm can be sufficiently reduced or the search direction is sufficiently large. Throughout Appendix \ref{ssec: pf-sec4}, we will frequently use the shorthand $\varepsilon_k=(\gamma_k\|\nabla f(x^k)\|)^{1/2}$ as defined in Algorithm \ref{alg:NCG-hd}.


\begin{lemma}\label{lem:either-case-sol}
Suppose Assumptions \ref{asp:basic} and \ref{asp:hd} hold, and that \texttt{CappedCG} outputs d$\_$type=SOL at iteration $k\geq0$ of Algorithm \ref{alg:NCG-hd}. If $\gamma_k\ge\gamma_\nu(\|\nabla f(x^k)\|)$, then either $\|\nabla f(x^{k+1})\|\leq \|\nabla f(x^k)\|/2$ or $11\|d^k\| \ge (\|\nabla f(x^k)\|/\gamma_k)^{1/2}$ holds.
\end{lemma}
\vspace{-5mm}
\begin{proof} As the algorithm does not terminate, we have $\|\nabla f(x^k)\|\neq0$ and hence $\varepsilon_k>0$. Because d\_type = SOL, we can see from the fourth inequality of Lemma \ref{lem:ppt-cg}(i) that $d^k\neq0$. If $11\|d^k\| \geq (\|\nabla f(x^k)\|/\gamma_k)^{1/2}$ holds, then the lemma is proved. Now, suppose $11\|d^k\| < (\|\nabla f(x^k)\|/\gamma_k)^{1/2}$. We next prove $\|\nabla f(x^{k+1})\|\le \|\nabla f(x^k)\|/2$ must hold. By Algorithm~\ref{alg:NCG-hd} Line 10, it suffices to show $f(x^k + d^k)\le f(x^k)$ and $\|\nabla f(x^k + d^k)\|\le \|\nabla f(x^k)\|/2$.

We first prove $f(x^k + d^k)\le f(x^k)$. Suppose for contradiction that $f(x^k + d^k) > f(x^k)$. Denote $\varphi(\alpha)=f(x^k + \alpha d^k)$. Then, one has $\varphi(1)>\varphi(0)$. Since $d^k\neq0$, it follows that
\begin{equation*}
\varphi^\prime(0) = \nabla f(x^k)^{\rm T} d^k \overset{(i)}{=} - (d^k)^{\rm T}(\nabla^2 f(x^k) + 2\varepsilon_k I) d^k \overset{(ii)}{\le} -\varepsilon_k\|d^k\|^2<0.    
\end{equation*}
Here, the steps (i) and (ii) follow from the third and first relations of Lemma \ref{lem:ppt-cg}, respectively,  with $\sigma = \varepsilon_k = (\gamma_k\|\nabla f(x^k)\|)^{1/2}$, according to the design of Algorithm \ref{alg:NCG-hd}. Together with  $\varphi(1)>\varphi(0)$, we know there must exists a local minimizer $\alpha_*$ such that  $\varphi^\prime(\alpha_*)=\nabla f(x^k + \alpha_* d^k)^{\rm T}d^k=0$ and $\varphi(\alpha_*)<\varphi(0)$. By Lemma \ref{lem:trial-xk-ccg-2}, we can then apply \eqref{ineq:desc-hd} to obtain that 
\begin{equation*}\begin{aligned}
\frac{H_f\alpha_*^{1+\nu}\|d^k\|^{2+\nu}}{1+\nu} & \ge (d^k)^{\rm T}\big(\nabla f(x^k + \alpha_* d^k) - \nabla f(x^k) - \alpha_* \nabla^2 f(x^k) d^k\big) \\ 
&  \overset{(i)}{=} - (d^k)^{\rm T}\nabla f(x^k) - \alpha_*(d^k)^{\rm T}\nabla^2 f(x^k) d^k\\
& \overset{(ii)}{=} (1-\alpha_*) (d^k)^{\rm T}\big(\nabla^2 f(x^k) + 2\varepsilon_k I\big) d^k + 2\alpha_*\varepsilon_k\|d^k\|^2\\
& \overset{(iii)}{\geq} \varepsilon_k\|d^k\|^2.    
\end{aligned}\end{equation*}
Here, (i) uses $\phi'(\alpha^*)=0$, (ii) and (iii) use the third and first relations of Lemma \ref{lem:ppt-cg}(i), respectively.  
As $\|d^k\|\neq0$, the above inequality further implies that 
\begin{equation*}
\|d^k\|^\nu \overset{(i)}{\ge}  \frac{(1+\nu)(\gamma_k\|\nabla f(x^k)\|)^{1/2}}{\alpha_*^{1+\nu}H_f}\overset{(ii)}{\ge} \frac{(\gamma_k\|\nabla f(x^k)\|)^{1/2}}{H_f} \overset{(iii)}{\geq} 2^{1+\nu}\bigg(\frac{\|\nabla f(x^k)\|}{\gamma_k}\bigg)^{\nu/2},    
\end{equation*}
where (i) uses the definition of $\varepsilon_k$ in Algorithm \ref{alg:NCG-hd}, (ii) uses $\alpha_*\in(0,1)$ and $\nu\geq0$, (iii) uses \eqref{ineq:tech-gam>} and $\gamma_k\geq\gamma_\nu(\|\nabla f(x^k)\|)$. Together with $11\|d^k\| < (\|\nabla f(x^k)\|/\gamma_k)^{1/2}$, we arrive at a contradiction that $11^\nu2^{1+\nu}<1$ with $\nu\in(0,1]$. Hence, we have proven the claim that $f(x^k + d^k)\le f(x^k)$. 

We next prove $\|\nabla f(x^k + d^k)\|\leq\|\nabla f(x^k)\|/2$. By Lemma \ref{lem:trial-xk-ccg-2}, we can apply \eqref{inexact-oracle}, the last inequality of Lemma \ref{lem:ppt-cg}(i), $\zeta_k\in(0,1)$, $11\|d^k\|<(\|\nabla f(x^k)\|/\gamma_k)^{1/2}$, and $\gamma_k\ge\gamma_\nu(\|\nabla f(x^k)\|)$ to obtain that 
\begin{equation*}\begin{aligned}
 \|\nabla f(x^k + d^k)\| & \le \mathcal{R}_1(x^k+ d^k,x^k) + \|(\nabla^2 f(x^k) + 2\varepsilon_kI)d^k + \nabla f(x^k)\|+ 2\varepsilon_k\|d^k\|\\ 
&\le \frac{\gamma_\nu(\|\nabla f(x^k)\|)}{4}\|d^k\|^2 + \frac{\|\nabla f(x^k)\|}{4} + \frac{5}{2}\varepsilon_k\|d^k\|\\
&\le \frac{\|\nabla f(x^k)\|}{484} + \frac{\|\nabla f(x^k)\|}{4} + \frac{5\|\nabla f(x^k)\|}{22} \le \frac{\|\nabla f(x^k)\|}{2}.
\end{aligned}\end{equation*} Hence, the proof of this lemma is complete.   
\end{proof}

With the above supporting lemmas, we can proceed to the proof of the main results in Section \ref{sec:ncg-hd}.

\subsubsection{Proof of Lemma \ref{lem:uni-sol-suf-desc}.}
\emph{Proof.} As the algorithm does not terminate, we have $\|\nabla f(x^k)\|\neq0$ and hence $\varepsilon_k>0$. Because d\_type = SOL, we can apply the fourth inequality of Lemma \ref{lem:ppt-cg}(i) to yield $d^k\neq0$. Moreover, for $k\in\mathbb{K}_{\epsilon,1}$, we have $\|\nabla f(x^{k+1})\|>\|\nabla f(x^k)\|/2$ and $\gamma_k> \gamma_{\nu}(\|\nabla f(x^k)\|)$. Using these and Lemma \ref{lem:either-case-sol}, we obtain that $11\|d^k\|\ge(\|\nabla f(x^k)\|/\gamma_k)^{1/2}$. With the above information, we can start the proof. \vspace{0.2cm}

\noindent\textbf{Statement (i).} \;If $\alpha_k=1$, the statement clearly holds. If $\alpha_k<1$,  then we know $j_k\geq 1$. In this case, for $j\in\{0,j_k-1\}$ that violates \eqref{ls-sol-stepsize-hd-}, we have
\begin{equation*}
\begin{aligned}
-\eta& \varepsilon_k\theta^j\|d^k\|^2 \!\le\! f(x^k\! +\! \theta^j d^k) \!- \!f(x^k)
\!\overset{(i)}{\le}\! \theta^j\nabla\! f(x^k)^{\rm T}d^k\! +\!\frac{\theta^{2j}}{2}(d^k)^{\rm \!T}\nabla^2\! f(x^k) d^k\! +\! \mathcal{R}_0(x^k \!+ \!\theta^j d^k\!,\!x^k)\\
\overset{(ii)}{\leq}&-\theta^j\Big(1-\frac{\theta^j}{2}\Big)(d^k)^{\rm T}\Big(\nabla^2 f(x^k)+2\varepsilon_kI\Big)d^k - \theta^{2j}\varepsilon_k\|d^k\|^2 + \frac{H_f\|\theta^jd^k\|^{2+\nu}}{2}\\
\overset{(iii)}{\le}& -\theta^j\varepsilon_k\|d^k\|^2 + \frac{H_f\theta^{(2+\nu)j}\|d^k\|^{2+\nu}}{2},
\end{aligned}
\end{equation*}  
where (i) is by the triangle inequality and the definition \eqref{def:R0}, (ii) is by the third relation of Lemma \ref{lem:ppt-cg}(i) and the residual bound \eqref{ineq:desc-hd}, and (iii) is by the first inequality of Lemma \ref{lem:ppt-cg}(i). As $d^k\neq0$, dividing both sides of the above inequality by $H_f\theta^{j}\|d^k\|^{2+\nu}/2$ yields
\begin{equation}
\label{eqn:theta-lower-unknown-nu-Sol}
\theta^{(1+\nu)j} \ge \frac{2(1-\eta)\varepsilon_k}{H_f\|d^k\|^\nu}  \geq  2^{2+\nu}(1-\eta)\frac{(\|\nabla f(x^k)\|/\gamma_k)^{\nu/2}}{\|d^k\|^\nu}, \qquad j\in\{0,j_k-1\}.  
\end{equation}
where the last inequality follows from the definition of $\varepsilon_k$, $\gamma_k>\gamma_\nu(\|\nabla f(x^k)\|)$, and \eqref{ineq:tech-gam>}.
Then setting $j=j_k-1$ gives
\begin{equation}\begin{aligned}\label{ineq:lwbd-ak-sol-alg2}
\alpha_k=\theta^{j_k}
&\ge 2(1-\eta)\theta\Big(\frac{(\|\nabla f(x^k)\|/\gamma_k)^{1/4}}{\|d^k\|^{1/2}}\Big)^{\frac{2\nu}{1+\nu}}\Big(\frac{(\|\nabla f(x^k)\|/\gamma_k)^{1/4}}{\sqrt{11}\|d^k\|^{1/2}}\Big)^{\frac{1-\nu}{1+\nu}} \\
&\ge \frac{(1-\eta)\theta(\|\nabla f(x^k)\|/\gamma_k)^{1/4}}{2\|d^k\|^{1/2}}\geq \frac{(1-\eta)\theta}{3}, 
\end{aligned}\end{equation}
where the first inequality uses $11\|d^k\|\ge(\|\nabla f(x^k)\|/\gamma_k)^{1/2}$ and  \eqref{eqn:theta-lower-unknown-nu-Sol}, and the last inequality uses the second inequality in  Lemma \ref{lem:ppt-cg}(i).  Hence the  statement (i) holds.\vspace{0.2cm}


\noindent\textbf{Statement (ii).} We prove this statement by considering two separate cases below. 

\noindent\textbf{Case 1.} $\alpha_k=1$. Recall that $11\|d^k\|\ge(\|\nabla f(x^k)\|/\gamma_k)^{1/2}$. Together with \eqref{ls-sol-stepsize}, this implies $f(x^k) - f(x^{k+1})\ge (\eta/121)\gamma_k^{-1/2}\|\nabla f(x^k)\|^{3/2}$. Hence, \eqref{ineq:uni-sol-suf-desc} follows by the definition of $c_{\rm sol}$.

\noindent\textbf{Case 2.} $\alpha_k<1$. In this case, by $11\|d^k\|\ge(\|\nabla f(x^k)\|/\gamma_k)^{1/2}$, \eqref{ls-sol-stepsize-hd-}, and \eqref{ineq:lwbd-ak-sol-alg2}, one has that $f(x^k) - f(x^{k+1})\ge\eta\varepsilon_k\alpha_k\|d^k\|^2\ge (\eta(1-\eta)\theta/400) \gamma_k^{-1/2}\|\nabla f(x^k)\|^{3/2}$,
which together with the definition of $c_{\rm sol}$ proves \eqref{ineq:uni-sol-suf-desc}. \vspace{0.2cm}

\noindent\textbf{Next gradient bound.}
By \eqref{inexact-oracle}, $\alpha_k\in(0,1]$, $\gamma_k\geq1$, $\zeta_k\le1$ (see Algorithm \ref{alg:NCG-hd}), $\gamma_\nu(\|\nabla f(x^k)\|)\le \gamma_k$, and the first, second and last inequalities of Lemma \ref{lem:ppt-cg}(i), one has
\vspace{-2mm}
\begin{equation*}\begin{aligned}
&\|\nabla f(x^{k+1})\| = \|\nabla f(x^k +\alpha_k d^k)\|\\
 \le& \,\mathcal{R}_1(x^k\!+\alpha_kd^k,x^k)\! + \alpha_k\|(\nabla^2 f(x^k)\! + 2\varepsilon_kI)d^k\! + \nabla f(x^k)\| \!+ 2\alpha_k\varepsilon_k\|d^k\|+(1-\alpha_k)\|\nabla f(x^k)\|\\
\le &\,\frac{\gamma_\nu(\|\nabla f(x^k)\|)}{4}\|d^k\|^2 + \frac{5}{4}\|\nabla f(x^k)\|\! +\frac{4+\zeta_k}{2}\varepsilon_k\|d^k\|\\
\le &\,\frac{\gamma_k}{4}1.1^2\frac{\|\nabla f(x^k)\|}{\gamma_k} + \frac{5}{4}\|\nabla f(x^k)\|\! +\frac{11}{4}\|\nabla f(x^k)\|\leq 5\|\nabla f(x^k)\|.
\end{aligned}\end{equation*}
\vspace{-2mm}
Hence, we complete the proof of this lemma. \hfill$\Box$


\subsubsection{Proof of Lemma \ref{lem:uni-nc-suf-desc}.} \emph{Proof.}
As the algorithm does not terminate, we know $\|\nabla f(x^k)\|\neq 0$. By Lemma \ref{lem:ppt-cg}(ii), we also confirm that $d^k\neq0$ in this case. We should also note that the search direction $d^k$ in this lemma is not the raw output of \texttt{CappedCG}. When d$\_$type=NC, it undergoes an additional re-normalization step (Algorithm \ref{alg:NCG-hd} Line 6) such that 
\begin{equation}\label{NCG-nu-2-ppty}
\nabla f(x^k)^{\rm T} d^k\le 0\qquad\mbox{and}\qquad \frac{(d^k)^{\rm T}\nabla^2 f(x^k) d^k}{\|d^k\|^2} = -\|d^k\| < -\varepsilon_k,
\end{equation}
where the second inequality is by Lemma \ref{lem:ppt-cg}(ii). Since $k\in\mathbb{K}_{\epsilon,1}$, we have that $\|\nabla f(x^{k+1})\| >\|\nabla f(x^k)\| / 2$ and $\gamma_\nu(\|\nabla f(x^k)\|)< \gamma_k$. With these in mind, let us prove the two statements.   \vspace{0.1cm}


\noindent\textbf{Statement (i).} If \eqref{ls-nc-stepsize-hd} holds $j=0$, then $\alpha_k=1$, which immediately implies the statement. Now let us
consider the case $j_k\geq 1$ and $\alpha_k=\theta^{j_k}<1$. Because \eqref{ls-nc-stepsize-hd} fails for $j=j_k-1$, one has
\begin{equation*}\begin{aligned}
-\frac{\eta}{2}\theta^{2j}\|d^k\|^3 & \!\le \!f(x^k\! +\! \theta^j d^k)\! -\! f(x^k) \!\le \theta^j \nabla f(x^k)^{\rm T} d^k \!+\! \frac{\theta^{2j}}{2} (d^k)^{\rm T} \nabla^2 f(x^k) d^k \!+\! \mathcal R_0(x^k\!+\!\theta^jd^k\!,x^k) \nonumber\\
& {\le} - \frac{\theta^{2j}}{2}\|d^k\|^3 + \frac{H_f\theta^{(2+\nu)j}}{2}\|d^k\|^{2+\nu}.
\end{aligned}\end{equation*}
where the second inequality is from \eqref{NCG-nu-2-ppty}. Next, dividing both sides by ${H_f\theta^{2j}\|d^k\|^{2+\nu}/2}$, setting $j=j_k-1$ in the above inequality and using the fact that $\nu\in(0,1],\eta\in(0,1/2]$ and $\|d^k\|\ge\varepsilon_k$ yields
\begin{equation}\begin{aligned}
\alpha_k=\theta^{ j_k} & \ge \theta\frac{ (1-\eta)^{1/\nu}\|d^k\|^{\frac{1-\nu}{\nu}}}{H_f^{1/\nu}} \ge \theta\frac{ (1-\eta)^{1/\nu}\varepsilon_k^{\frac{1-\nu}{\nu}}}{H_f^{1/\nu}} \ge \theta\frac{ (1-\eta)^{1/\nu}2^{\frac{1+\nu}{\nu}}}{\gamma_k} \ge \theta/\gamma_k,\label{lwbd-nc-step-thetaj}
\end{aligned}\end{equation}
where the third inequality follows from $\varepsilon_k=(\gamma_k\|\nabla f(x^k)\|)^{1/2}$ and \eqref{ineq:tech-gam>}. Hence, statement (i) holds.

\noindent\textbf{Statement (ii).} By \eqref{ls-nc-stepsize-hd}, \eqref{NCG-nu-2-ppty}, and \eqref{lwbd-nc-step-thetaj} we have 
\begin{equation*}
f(x^k) - f(x^{k+1}) \ge \frac{\eta}{2}\alpha_k^2\|d^k\|^3\ge\frac{\eta\theta^2}{2}\gamma_k^{-\frac{1}{2}}\|\nabla f(x^k)\|^{\frac{3}{2}},    
\end{equation*}
which together with the definition of $c_{\rm nc}$ implies that \eqref{lwbd:f-descent-hd-nc} holds as desired.    \vspace{0.2cm}

\noindent\textbf{Next gradient bound.} Given $\|d^k\|\leq\theta U_H$, we consider two separate cases. When $\alpha_k=1$, we have $U_H\|d^k\|\!\geq\!\theta U_H\|d^k\|\!\geq\!\|d^k\|^{2}/\gamma_k.$ When $\alpha_k\!<\!1$, by \eqref{lwbd-nc-step-thetaj} we have that $\alpha_kU_H\|d^k\|\!\geq\! \frac{\theta}{\gamma_k} U_H\|d^k\|\!\geq \!\frac{\|d^k\|^{2}}{\gamma_k}.$
In addition, by Assumption \ref{asp:basic} and \ref{asp:hd}, one has $\nabla f$ is Lipschitz continuous, which implies that 
\begin{equation}\label{ineq:upbd gk+1 nc}
    \|\nabla f(x^{k+1})\|\leq \|\nabla f(x^k)\|+\alpha_kU_H\|d^k\|{\le} \frac{\|d^k\|^2}{\gamma_k}+\alpha_k U_H\|d^k\|\leq2\alpha_kU_H\|d^k\|,
\end{equation}
where the second inequality is due to \eqref{NCG-nu-2-ppty} and $\varepsilon_k=(\gamma_k\|\nabla f(x^k)\|)^{1/2}$. This completes the proof.\hfill$\Box$

\subsubsection{Proof of Lemma \ref{lem:upbd-gammak}.} \emph{Proof.}
Suppose for contradiction that $\gamma_k>\max\{\gamma_0,2\gamma_\nu(\epsilon)\}$ for some $k\in\mathbb{K}_\epsilon$. Then there exists $\tilde{k}\le k-1$ such that $\max\{\gamma_0,2\gamma_\nu(\epsilon)\}<\gamma_{\tilde{k}+1}=2\gamma_{\tilde{k}}$. By this and $\|\nabla f(x^{\tilde{k}})\|>\epsilon$, one has that $\gamma_{\tilde{k}}>\gamma_\nu(\epsilon)\ge\gamma_\nu(\|\nabla f(x^{\tilde{k}})\|)$. In addition, since $\gamma_{\tilde{k}+1}=2\gamma_{\tilde{k}}$, it follows from Algorithm \ref{alg:NCG-hd} that $\|\nabla f(x^{\tilde{k}+1})\|>\|\nabla f(x^{\tilde{k}})\|/2$, and either $f(x^{\tilde{k}})-f(x^{\tilde{k}+1})<c_{\mathrm{sol}}\gamma_{\tilde{k}}^{-1/2}\|\nabla f(x^{\tilde{k}})\|^{3/2}$ 
when d\_type = SOL, or $\alpha_{\tilde k} <\theta/\gamma_{\tilde k}$ when d\_type = NC. Note that $\tilde{k}\in\mathbb{K}_{\epsilon,1}$, which leads to a contradiction with the sufficient descent established in Lemma \ref{lem:uni-sol-suf-desc} and the lower bound for $\alpha_{\tilde{k}}$ established in Lemma~\ref{lem:uni-nc-suf-desc}. Hence, we have $\gamma_k\le\max\{\gamma_0,2\gamma_\nu(\epsilon)\}$ for all $k\in\mathbb{K}_\epsilon$.\hfill$\Box$



\subsubsection{Proof of Theorem~\ref{thm:glb-cplx-2}.}\label{appdx:Thm-2}  \emph{Proof. }  Recall that $\mathbb{K}_\epsilon=\{k:\|\nabla f(x^t)\|>\epsilon, \forall t\leq k\}$ contains all iterations before finding an $\epsilon$-stationary point. Also, $|\mathbb{K}_{\epsilon,i}|$, $i=1,2,3$, form a partition of $\mathbb{K}_{\epsilon}$. \vspace{0.1cm}

\noindent\textbf{Part 1. } Bounding $|\mathbb{K}_{\epsilon,1}|$. 
Combining Lemmas \ref{lem:uni-sol-suf-desc} , \ref{lem:uni-nc-suf-desc} and \ref{lem:upbd-gammak}, we know that each iteration $k\in\mathbb{K}_{\epsilon,1}$ results in a sufficient descent of 
$$f(x^k)-f(x^{k+1})\ge \min\{c_{\mathrm{sol}},c_{\mathrm{nc}}\} \epsilon^{\frac{3}{2}}/\gamma_k^{\frac{1}{2}}\ge\min\{c_{\mathrm{sol}},c_{\mathrm{nc}}\} \epsilon^{\frac{3}{2}}/\max\{\gamma_0, 2\gamma_\nu(\epsilon)\}^{\frac{1}{2}},$$
As Algorithm \ref{alg:NCG-hd} is a descent method and $\Delta_f = f(x^0) - f_{\rm low}$, by the definition of $\gamma_{\nu}(\cdot)$, one has
\begin{equation*}
|\mathbb{K}_{\epsilon,1}|\le \frac{\Delta_f\epsilon^{-\frac{3}{2}}\max\{\gamma_0, 2\gamma_\nu(\epsilon)\}^{\frac{1}{2}}}{\min\{c_{\mathrm{sol}},c_{\mathrm{nc}}\}} = \frac{\Delta_f\max\{\gamma_0^{\frac{1}{2}}\epsilon^{-\frac{3}{2}},2\sqrt{2}H_f^{\frac{1}{1+\nu}}\epsilon^{-\frac{2+\nu}{1+\nu}}\}}{\min\{c_{\mathrm{sol}},c_{\mathrm{nc}}\}} = \mathcal{O}\left(\Delta_fH_f^{\frac{1}{1+\nu}}\epsilon^{-\frac{2+\nu}{1+\nu}}\right) .\end{equation*}

\noindent\textbf{Part 2. } Bounding $|\mathbb{K}_{\epsilon,2}|$. 
It follows from the update rule of $\{\gamma_k\}$, and  Lemma \ref{lem:uni-sol-suf-desc} and \ref{lem:uni-nc-suf-desc} that $\gamma_k$ increases only for some $k\in\mathbb K_{\epsilon,2}$. Thus, we  further divide  $\mathbb K_{\epsilon,2}$ into  $\mathbb K_{\epsilon,2}^1:=\{k\in\mathbb K_{\epsilon,2}:\gamma_{k+1}=2\gamma_k\}$ and $\mathbb K_{\epsilon,2}^2:=\mathbb{K}_{\epsilon,2}\setminus \mathbb K_{\epsilon,2}^1$. For $k\in \mathbb K_{\epsilon,2}^1 $, $\gamma_{k+1}=2\gamma_k$. Since $\{\gamma_k\}_{k\geq0}$ is nondecreasing and is always bounded above by $\max\{\gamma_0,2\gamma_\nu(\epsilon)\}$, we have
$ |\mathbb K_{\epsilon,2}^1|\leq \lceil \log_2( \max\{\gamma_0,2\gamma_\nu(\epsilon)\}/ \gamma_0)\rceil\leq\mathcal O(\ln(1/\epsilon)).$ For $k\in \mathbb K_{\epsilon,2}^2, $ one has  $\|\nabla f(x^{k+1})\|>\|\nabla f(x^k)\|/2$. In addition, since $\gamma_{k+1}=\gamma_k$, by Lines 7 and 12 of Algorithm \ref{alg:NCG-hd}, we have either a sufficient descent $f(x^k) - f(x^{k+1}) \ge c_{\rm sol}\gamma_k^{-1/2}\|\nabla f(x^k)\|^{1/2}$ for d\_type=SOL, or a lower step size bound $\alpha_k\geq \theta/\gamma_k$ for d\_type=NC, which further implies a sufficient descent \(f(x^k) - f(x^{k+1}) \ge c_{\rm nc}\gamma_k^{-1/2}\|\nabla f(x^k)\|^{1/2}\) due to  \eqref{ls-nc-stepsize-hd} and \eqref{NCG-nu-2-ppty}. Therefore,    each iteration $k\in\mathbb{K}_{\epsilon,2}^2$ results in a sufficient descent in the objective value. Using the same argument in Part 1, we obtain that $ |\mathbb K_{\epsilon,2}^2|\le\mathcal{O}\Big(\Delta_fH_f^{\frac{1}{1+\nu}}\epsilon^{-\frac{2+\nu}{1+\nu}}\Big)$. Combining the upper bounds of $|\mathbb{K}_{\epsilon,2}^1|$ and $|\mathbb{K}_{\epsilon,2}^2|$ yields
   $ |\mathbb{K}_{\epsilon,2}| \leq \mathcal O(\Delta_f H_f^{\frac{1}{1+\nu}}\epsilon^{-\frac{2+\nu}{1+\nu}}).$ \vspace{0.2cm}

\noindent\textbf{Part 3.} Bounding $|\mathbb{K}_{\epsilon,3}|$.
For each $k\in\mathbb{K}_{\epsilon,3}$, the next gradient halves: $\|\nabla f(x^{k+1})\|\leq \|\nabla f(x^{k})\|/2$. Note that $\mathbb{K}_{\epsilon,1}$, $\mathbb{K}_{\epsilon,2}^1$, $\mathbb{K}_{\epsilon,2}^2$ are finite and $\mathbb{K}_{\epsilon}=\mathbb{K}_{\epsilon,1}\cup \mathbb{K}_{\epsilon,2}^1\cup \mathbb{K}_{\epsilon,2}^2\cup \mathbb{K}_{\epsilon,3}$, one can see that $\mathbb{K}_{\epsilon,3}$ can be partitioned into $|\mathbb{K}_{\epsilon,1}|+|\mathbb{K}_{\epsilon,2}^1|+|\mathbb{K}_{\epsilon,2}^2|+1$ disjoint subsets of \emph{consecutive} nonnegative integers.  Now, for those subsets that follow an iteration index $k\in \mathbb{K}_{\epsilon,1}$, we denote them by $\mathbb{K}_{\epsilon,3}^{i}$ with $i\in\mathcal I_1:=\{1,\cdots,\ell_1\}$. For those subsets that follow an iteration index $k\in\mathbb{K}^2_{\epsilon,2}$, we denote them by $\mathbb{K}_{\epsilon,3}^{i}$ with $i\in \mathcal I_2=\{\ell_1+1,\cdots,\ell_1+\ell_2\}$. For those subsets that follow an iteration index $k\in\mathbb{K}^1_{\epsilon,2}$, or in the case where the subset contains $0$, we denote them by $\mathbb{K}_{\epsilon,3}^{i}$ with $i\in\mathcal I_3:=\{\ell_1+\ell_2+1,\cdots,\ell_1+\ell_2+\ell_3\}$. Then clearly, $\ell_1\leq |\mathbb{K}_{\epsilon,1}|$, $\ell_2\leq |\mathbb{K}^2_{\epsilon,2}|$, and $\ell_3\leq |\mathbb{K}_{\epsilon,2}^1|+1$. 
\vspace{0.2cm}

\noindent \textbf{Part 3 (i).} Bounding $\sum_{i\in\mathcal{I}_3} |\mathbb{K}_{\epsilon,3}^{i}|$. Since $\epsilon\leq\|\nabla f(x^k)\|\leq U_g$ for all $k\in\mathbb{K}_{\epsilon}$, each subset has at most $\lceil\ln (U_g/\epsilon)/\ln 2\rceil+1$ iterations. Consequently, for those subsets following an $\mathbb{K}_{\epsilon,2}^1$ iteration, we have 
\vspace{-2mm}
\begin{equation*}
\sum_{i\in\mathcal{I}_3} |\mathbb{K}_{\epsilon,3}^{i}|\leq \left(|\mathbb{K}_{\epsilon,2}^1|+1\right)\left(\left\lceil\frac{\ln (U_g/\epsilon)}{\ln 2}\right\rceil+1\right) \leq  \mathcal{O}\left(\left(\ln\frac{1}{\epsilon}\right)^2\right).
\end{equation*}

For the rest subsets, directly applying the above bound would result in an undesirable $\mathcal O(\epsilon^{-\frac{2+\nu}{1+\nu}}\ln\frac{1}{\epsilon})$ complexity. To remove the extra logarithmic factor, a more careful analysis is required. For notational simplicity, let us suppose that the initial iteration of each $\mathbb{K}_{\epsilon,3}^i$ has gradient $\delta_i$, and the iteration  prior to each $\mathbb{K}^i_{\epsilon,3}$ has gradient norm $\bar{\delta}_i$, direction norm $\hat{\delta}_i$ and line-search step length $\hat\alpha_i$.\vspace{0.2cm}

\noindent \textbf{Part 3 (ii).} Bounding $\sum_{i\in\mathcal{I}_1} |\mathbb{K}_{\epsilon,3}^{i}|$. We further partition  $\mathcal I_1=\mathcal I_{\rm sol} \cup\mathcal I_{\rm nc}$ according to whether the subset follows an iteration $k\in \mathbb{K}_{\epsilon,1}$ with d\_type=SOL or NC, respectively. 

To bound $\sum_{i\in\mathcal{I}_{\rm sol}}|\mathbb{K}^i_{\epsilon,3}|$, by the relationship $\delta_i\le 5\bar\delta_i$ from Lemma \ref{lem:uni-sol-suf-desc}, we have \begin{equation*}
\begin{aligned}
\sum_{i\in\mathcal{I}_{\rm sol}}|\mathbb{K}_{\epsilon,3}^{i}| \leq & \sum_{i\in\mathcal{I}_{\rm sol}} \left(\left\lceil\frac{\ln (\delta_i/\epsilon)}{\ln 2}\right\rceil+1\right) 
 {\le}  (2+ 2\ln(5))|\mathcal{I}_{\rm sol}| +\sum_{i\in\mathcal{I}_{\rm sol}} 2\ln (\bar\delta_i/\epsilon) \\
 \leq& 6|\mathbb{K}_{\epsilon,1}| + \frac{4\Delta_f\max\{\gamma_0, 2\gamma_\nu(\epsilon)\}^{1/2}}{3ec_{\mathrm{sol}}\epsilon^{3/2}} \leq \mathcal{O}\left(\Delta_fH_f^{\frac{1}{1+\nu}}\epsilon^{-\frac{2+\nu}{1+\nu}}\right) 
\end{aligned}\end{equation*}
and the third inequality uses $|\mathcal I_{\rm sol}|\leq|\mathbb K_{\epsilon,1}|$, $\sum_{i\in\mathcal{I}_{\rm sol}}\bar{\delta}_i^{3/2}\le \Delta_f\max\{\gamma_0,2\gamma_\nu(\epsilon)\}^{1/2}/c_{\mathrm{sol}}$ (due to \eqref{ineq:uni-sol-suf-desc} and Lemma \ref{lem:upbd-gammak}), and Lemma \ref{lem: sum log upbd} with $(z_i,p,q,c)=(\bar\delta_i,3/2,3/2,\epsilon)$.

For $\sum_{i\in\mathcal{I}_{\rm nc}}|\mathbb{K}^i_{\epsilon,3}|$, we can further divide $\mathcal{I}_{\rm nc}$ into  $\mathcal{I}_{\rm nc}^{\leq}:=\{i\in\mathcal{I}_{\rm nc}:\hat\delta_i\leq \theta U_H\}$, and $\mathcal{I}_{\rm nc}^{>}:=\mathcal{I}_{\rm nc}\setminus\mathcal{I}_{\rm nc}^{\leq}$. For each $i\in \mathcal{I}_{\rm nc}^>$, by \eqref{ls-nc-stepsize-hd}, we know the last iteration in $|\mathbb K_{\epsilon,1}|$ will guarantee a descent of at least $$\frac{\eta}{2}\hat\alpha_i^2\hat\delta_i^3\geq\frac{\eta\theta^5U_H^3}{2\gamma_k^2}\geq\frac{\eta\theta^5U_H^3}{2\max\{\gamma_0,2\gamma_\nu(\epsilon)\}^2}\geq\frac{\eta\theta^5U_H^3}{2\max\{\gamma_0^2,64H_f^{4/(1+\nu)}\epsilon^{-(2-2\nu)/(1+\nu)}\}},$$
where the first inequality follows from $\hat\alpha_k\geq\theta/\gamma_k$, and the second inequality is because of $\gamma_k\leq\max\{\gamma_0,2\gamma_{\nu}(\epsilon)\}$. As the total
descent is controlled by $\Delta_f$, we have $|\mathcal{I}_{\rm nc}^{>}|\leq\mathcal O(\Delta_f\epsilon^{-\frac{2-2\nu}{1+\nu}})$ and
\begin{equation*}
\sum_{i\in\mathcal{I}_{\rm nc}^>} |\mathbb{K}_{\epsilon,3}^{i}|\leq |\mathcal{I}_{\rm nc}^>|\cdot\left(\left\lceil\frac{\ln (U_g/\epsilon)}{\ln 2}\right\rceil+1\right) \leq o\left(\epsilon^{-\frac{2+\nu}{1+\nu}}\right).
\end{equation*}
For $i\in \mathcal{I}_{\rm nc}^\leq$,  by \eqref{ls-nc-stepsize-hd} we know the last iteration,  indexed by $k_{i}$, in $|\mathbb K_{\epsilon,1}|$ will guarantee a descent of at least $\frac{\eta}{2}\hat\alpha_i^2\hat\delta_i^3\geq \frac{\eta}{2}\hat\alpha_i^2\hat\delta_i^2\epsilon^{\frac12}$, where we single out one $\hat\delta_i$ and lower bound it by \eqref{NCG-nu-2-ppty} under the fact that $\hat{\delta}_i\geq \varepsilon_{k_i} = (\gamma_{k_i}\|\nabla f(x^{k_i})\|)^{\frac12}\geq\epsilon^{\frac{1}{2}}$. Again, using $\Delta_f$ to control the total descent gives 
\begin{equation}
    \label{eqn:NC-nu-upperDescent-alg2}
    \sum_{i\in\mathcal{I}_{\rm nc}^{\leq}}\hat\alpha_i^2\hat\delta_i^2\leq\frac{2\Delta_f}{\eta\cdot \epsilon^{\frac{1}{2}}},
\end{equation}
which shall be used later as a summation upper bound in Lemma \ref{lem: sum log upbd}. With this in mind and using the inequality $\delta_i\leq 2\hat{\alpha}_iU_H\hat{\delta}_i$ provided by Lemma \ref{lem:uni-nc-suf-desc}, we have \begin{equation*}
\begin{aligned}
\sum_{i\in\mathcal{I}^\leq_{\rm nc}}|\mathbb{K}_{\epsilon,3}^{i}| \leq & \sum_{i\in\mathcal{I}^\leq_{\rm nc}} \left(\left\lceil\frac{\ln (\delta_i/\epsilon)}{\ln 2}\right\rceil+1\right) 
 {\le} \left(2+ 2\ln(2U_H)\right)|\mathcal{I}^\leq_{\rm nc}|+2\sum_{i\in\mathcal{I}^\leq_{\rm nc}} \ln ((\hat\alpha_i\hat\delta_i)/\epsilon)\\
{\le} & (2+2\ln(2U_H))|\mathbb K_{\epsilon,1}|+
\frac{4\Delta_f}{e\eta\epsilon} \leq \mathcal{O}\left(\Delta_fH_f^{\frac{1}{1+\nu}}\epsilon^{-\frac{2+\nu}{1+\nu}}\right) 
\end{aligned}\end{equation*}
where the third inequality follows from $|\mathcal{I}^\leq_{\rm nc}|\leq|\mathbb{K}_{\epsilon,1}|$, the summation bound \eqref{eqn:NC-nu-upperDescent-alg2}, and  Lemma \ref{lem: sum log upbd} with $(z_i,p,q,c)=(\hat\alpha_i\hat\delta_i,2,1/2,\epsilon)$.\vspace{0.2cm}

\noindent\textbf{Part 3 (iii). } Bounding $\sum_{i\in \mathcal I_2}|\mathbb{K}^i_{\epsilon,3}|$. We partition $\mathcal I_2=\mathcal I'_{\rm sol} \cup\mathcal I'_{\rm nc}$ according to whether the subset follows an iteration  $k\in \mathbb{K}_{\epsilon,2}^2$ with d\_type=SOL or NC, respectively. Before proceeding, we establish an upper bound for the next gradient norm with iteration index $k\in \mathbb{K}_{\epsilon,2}^2$. For d\_type=SOL, by \eqref{inexact-oracle},  the fact that $\alpha_k\leq1$, $\gamma_k\geq1$, and $\zeta_k\leq 1/2$ in Algorithm \ref{alg:NCG-hd}, and the first, second and last relations of Lemma \ref{lem:ppt-cg}(i), one has 
{
\setlength{\abovedisplayskip}{1pt}
\setlength{\belowdisplayskip}{0pt}
\begin{align}
&\|\nabla f(x^{k\!+\!1})\| = \|\nabla f(x^k \!+\!\alpha_k d^k)\|\notag\\
 \!\le&\, \mathcal{R}_1(x^{k}\!+\!\alpha_kd^k,x^k) \!+\! \alpha_k\|(\nabla^2 f(x^k) \!+\! 2\varepsilon_kI)d^k \!+\! \nabla f(x^k)\| \!+\! 2\alpha_k\varepsilon_k\|d^k\|\!+\!(1-\alpha_k)\|\nabla f(x^k)\|\notag\\
\!\le&\, \frac{\gamma_\nu(\|\nabla f(x^k)\|)}{4}\|d^k\|^2 \!+\! \frac{5}{4}\|\nabla f(x^k)\| \!+\!\frac{4\!+\!\zeta_k}{2}\varepsilon_k\|d^k\|
\!\le 2H_f^{\frac{2}{1+\nu}}\|\nabla f(x^k)\|^{\frac{2\nu}{1+\nu}}\!+\!6\|\nabla f(x^k)\|.\label{ineq: g_k+1 k22}
\end{align}
}

\noindent Thus, we have $\|\nabla f(x^{k+1})\|\leq  4H_f^{\frac{2}{1+\nu}}\|\nabla f(x^k)\|^{\frac{2\nu}{1+\nu}}$ if $\|\nabla f(x^k)\|\leq(H_f^{2}/3^{1+\nu})^{1/(1-\nu)}$ and $\nu\in(0,1)$.
For d\_type = NC, the same resoning in \eqref{ineq:upbd gk+1 nc} gives \(\|\nabla f(x^{k+1})\| \leq 2\alpha_kU_H\|d^k\|\) whenever  $\|d^k\|\leq \theta U_H$. With these next gradient bounds, we are ready to bound $\sum_{i\in\mathcal{I}'_{\rm sol}}|\mathbb{K}_{\epsilon,3}^{i}|$ and $\sum_{i\in\mathcal{I}'_{\rm nc}}|\mathbb{K}_{\epsilon,3}^{i}|$.

For $\sum_{i\in\mathcal{I}'_{\rm sol}}|\mathbb{K}_{\epsilon,3}^{i}|$, we consider the following two cases:

    \noindent \textbf{Case 1.} When $\nu\in(0,1)$,
 we further divide $\mathcal I'_{\rm sol}$ into  $\mathcal I_{\rm sol}^{'\leq}:=\{i\in\mathcal I'_{\rm sol}:\bar\delta_{i}\leq (H_f^{2}/3^{1+\nu})^{1/(1-\nu)}\}$ and $\mathcal I^{'>}_{\rm sol}:= \mathcal I'_{\rm sol}\setminus\mathcal I^{'\leq}_{\rm sol}$. 
For $\mathcal I_{\rm sol}^{'>}$, we have $\bar\delta_i> (H_f^{2}/3^{1+\nu})^{1/(1-\nu)}$ . By Line 12 of Algorithm \ref{alg:NCG-hd}, we know the last iteration in $|\mathbb K_{\epsilon,2}^2 |$, indexed by $k$, will guarantee a descent of at least $$c_{\mathrm{sol}}\gamma_k^{-1/2}\|\nabla f(x^k)\|^{3/2}\geq c_{\rm sol}\max\{\gamma_0,2\gamma_\nu(\epsilon)\}^{-1/2}H_f^{3/(1-\nu)}/3^{(3+3\nu)/(2-2\nu)}\geq \Omega(\epsilon^{\frac{1-\nu}{2+2\nu}}).$$ Using $\Delta_f$ to control the total descent yields  $|\mathcal I_{\rm sol}^{'>}|\leq \mathcal O(\Delta_f\epsilon^{-\frac{1-\nu}{2+2\nu}})$, which implies 
\begin{equation*}
\sum_{i\in\mathcal{I}_{\rm sol}^{'>}} |\mathbb{K}_{\epsilon,3}^{i}|\leq |\mathcal I_{\rm sol}^{'>}|\cdot\left(\left\lceil\frac{\ln (U_g/\epsilon)}{\ln 2}\right\rceil+1\right) \leq o\left(\epsilon^{-\frac{2+\nu}{1+\nu}}\right).
\end{equation*}
For $\mathcal I_{\rm sol}^{'\leq}$, the next gradient bound for $|\mathbb{K}_{\epsilon,2}^{2}|$ implies $\delta_i\leq4H_f^{2/(1+\nu)}\bar\delta_i^{2\nu/(1+\nu)}$.
It then follows that 
\begin{equation*}
\begin{aligned}
\sum_{i\in\mathcal{I}'^\leq_{\rm sol}}|\mathbb{K}_{\epsilon,3}^{i}| \leq & \sum_{i\in\mathcal{I}_{\rm sol}'^{\leq}} \left(\left\lceil\frac{\ln (\delta_i/\epsilon)}{\ln 2}\right\rceil+1\right) 
 \leq  \left(2+ 2\ln(4H_f^{2/(1+\nu)})\right)|\mathcal{I}_{\rm sol}'^{\leq}| +2\sum_{i\in\mathcal{I}_{\rm sol}^{'\leq}} \ln ({\bar\delta_i^{2\nu/(2+\nu)}}/\epsilon)\\
 \leq& (2+2\ln(4H_f^{2/(1+\nu)}))|\mathbb{K}_{\epsilon,2}^2| + {\frac{4\Delta_f\max\{\gamma_0,2\gamma_\nu(\epsilon)\}^{1/2}}{ec_{\mathrm{sol}}\epsilon^{1/2}}} 
 \leq\mathcal O(\Delta_fH_f^{\frac{1}{1+\nu}}\epsilon^{-\frac{2+\nu}{1+\nu}}),\\
\end{aligned}\end{equation*}
where the third inequality follows from $|\mathcal{I}^{'\leq}_{\rm sol}|\leq|\mathbb{K}_{\epsilon,2}^2|$, {$\sum_{i\in\mathcal{I}_{\rm sol}^{'\leq}}\bar\delta_i^{3/2}\le \Delta_f\max\{\gamma_0,2\gamma_\nu(\epsilon)\}^{1/2}/c_{\rm sol}$} (due to the definitions of $\mathcal{I}_2$ and $\mathbb{K}_{\epsilon,2}^2$) and Lemma \ref{lem: sum log upbd} with {$(z_i,p,q,c)=(\bar\delta_i^{\frac{2\nu}{2+\nu}},\frac{3(2+\nu)}{4\nu},\frac{1}{2},\epsilon)$}.

\noindent\textbf{Case 2.} When $\nu=1$, \eqref{ineq: g_k+1 k22} implies $\delta_i\leq (2H_f+6)\bar\delta_i$. Then we have
\begin{equation*}
\begin{aligned}
\sum_{i\in\mathcal{I}'_{\rm sol}}|\mathbb{K}_{\epsilon,3}^{i}|& \leq   \sum_{i\in\mathcal{I}'_{\rm sol}} \left(\left\lceil\frac{\ln (\delta_i/\epsilon)}{\ln 2}\right\rceil+1\right)    
{\le}  2(1+\ln(2H_f+6))|\mathcal{I}'_{\rm sol}| +2\sum_{i\in\mathcal{I}_{\rm sol}} \ln (\bar\delta_i/\epsilon) \\ 
\leq& 2\left(1+\ln(2H_f+6)\right)|\mathbb{K}_{\epsilon,2}^2| + {\frac{2\Delta_f\max\{\gamma_0,8H_f\}^{1/2}}{ec_{\mathrm{sol}}\epsilon}} 
= \mathcal{O}(\Delta_fH_f^{\frac{1}{2}}\epsilon^{-\frac{3}{2}}),  
\end{aligned}\end{equation*}
where the third inequality uses $|\mathcal{I}'_{\rm sol}|\leq|\mathbb{K}_{\epsilon,2}^2|$, $\sum_{i\in\mathcal{I}_{\rm sol}^{'\leq}}\bar\delta_i^{3/2}\le {\Delta_f\max\{\gamma_0,8H_f\}^{1/2}/c_{\rm sol}}$ (due to the definitions of $\mathcal{I}_2$ and $\mathbb{K}_{\epsilon,2}^2$) and Lemma \ref{lem: sum log upbd} with $(z_i,p,q,c)=(\bar\delta_i,3/2,1,\epsilon)$.   

Combining the above two cases, we obtain that $\sum_{i\in\mathcal{I}'_{\rm sol}}|\mathbb{K}_{\epsilon,3}^{i}|\leq O(\Delta_fH_f^{\frac{1}{1+\nu}}\epsilon^{-\frac{2+\nu}{1+\nu}})$ for all $\nu\in(0,1]$. Note from the proof of Lemma \ref{lem:uni-nc-suf-desc} that the result $\|\nabla f(x^{k+1})\|\le2\alpha_k U_H\|d^k\|$ whenever $\|d^k\|\le \theta U_H$ is also valid at iteration $k\in\mathbb{K}_{\epsilon,2}^2$ because $\alpha_k\ge\theta/\gamma_k$. Moreover, by the definition of $\mathcal I'_{\rm nc}$, one sees that $\mathbb{K}_{\epsilon,2}^i$, $i\in\mathcal I'_{\rm nc}$ follows an iteration $k\in\mathbb{K}_{\epsilon,2}^2$. Therefore, for $\mathcal I'_{\rm nc}$, we can use the same arguments as those for bounding $\sum_{i\in\mathcal{I}_{\rm nc}}|\mathbb{K}^i_{\epsilon,3}|$ in Part 3(ii). That is, we start by dividing $\mathcal I'_{\rm nc}$ into two disjoint subsets: $\mathcal{I}^{',\le}_{\rm nc}=\{i\in\mathcal I'_{\rm nc}:\hat{\delta}_i\le \theta U_H\}$ and $\mathcal{I}^{',>}_{\rm nc}=\mathcal{I}^{'}_{\rm nc}/\mathcal{I}^{',\le}_{\rm nc}$. We then bound $\sum_{i\in\mathcal{I}^{',\le}_{\rm nc}}|\mathbb{K}_{\epsilon,3}^i|$ and $\sum_{i\in\mathcal{I}^{',>}_{\rm nc}}|\mathbb{K}_{\epsilon,3}^i|$, respectively, following the same derivations used in Part 3(ii). As a result, we obtain $\sum_{i\in\mathcal{I}'_{nc}}|\mathbb{K}^i_{\epsilon,3}|\leq O(\Delta_fH_f^{\frac{1}{1+\nu}}\epsilon^{-\frac{2+\nu}{1+\nu}})$.


Lastly, summarizing Parts 1-3, we obtain that
\vspace{-1mm}
\begin{equation*}
    \begin{aligned}
        |\mathbb{K}_{\epsilon}|& =|\mathbb{K}_{\epsilon,1}|+|\mathbb{K}_{\epsilon,2}|+\sum_{i\in\mathcal{I}_{3}}|\mathbb{K}_{\epsilon,3}^i|+\sum_{i\in\mathcal{I}_{\rm sol}}|\mathbb{K}_{\epsilon,3}^i|+\sum_{i\in\mathcal{I}_{\rm nc}}|\mathbb{K}_{\epsilon,3}^i|+\sum_{i\in\mathcal{I}'_{\rm sol}}|\mathbb{K}_{\epsilon,3}^i|+\sum_{i\in\mathcal{I}'_{\rm nc}}|\mathbb{K}_{\epsilon,3}^i|\\
        &\leq\mathcal{O}\left(\Delta_fH_f^{\frac{1}{1+\nu}}\epsilon^{-\frac{2+\nu}{1+\nu}}\right).
    \end{aligned}
\end{equation*}
Hence, we complete the proof of this theorem. 
\hfill $\Box$


\subsubsection{Proof of Lemma \ref{lemma:gamma_bound_universal}.}\label{proof: gamma_bound_universal}\emph{Proof. }
First, let us define the radius constant $\delta$ in Theorem \ref{thm:loc-conv-2} as
\begin{equation}
    \label{eqn:defn-delta-Thm4}
    \delta=\min\bigg\{\frac{1}{2L_gH_f^{1/\nu}},\bigg(\frac{\mu}{2H_f}\bigg)^{\frac1\nu},\frac{1}{2L_g}\bigg(\frac{\mu}{2H_f}\bigg)^{\frac2\nu},\bigg[\frac{6L_g}{\mu}\bigg(\frac{2H_f}{\mu}\!+\!8\bigg(\frac{H_f}{\mu}\bigg)^{\frac{1}{1\!+\!\nu}}\!+\!L_g^{\frac{1}{2}}\bigg)\bigg]^{-\frac{1+\nu}{\nu}}\bigg\},
\end{equation} 
where $L_g:=\|\nabla^2 f(x^\ast)\|+1$. The same argument at the beginning of \ref{appdx:Thm-local-Known-Nu} yields that $B_{\delta}(x^\ast)\subseteq\mathscr{L}_f(x^0;r_d)$, the relation \eqref{eqn:loc-SC} holds for all $x\in B_{\delta}(x^\ast)$,
 once an iteration enters the $B_{\delta}(x^*)$ neighborhood, Algorithm \ref{alg:NCG-hd} will always execute the Newton step with $\alpha_k=1$ and d\_type = SOL.

We now show that $\gamma_{k+1}=\gamma_k$ for any $x^k\in B_{\delta}(x^*)$. From Algorithm  \ref{alg:NCG-hd} we know that $\gamma_k$ increases only when $k\in\mathbb K_{\epsilon,2}$. To prove this statement, it is suffice to show that  $k\notin\mathbb K_{\epsilon,2}$ when $x^{ k}\in B_{\delta}(x^*)$. Suppose for contradiction that $x^{ k}\in B_{\delta}(x^*)$ for some $ k\in\mathbb K_{\epsilon,2}$. 
Then, we have
{\setlength{\abovedisplayskip}{.5pt}
\setlength{\belowdisplayskip}{0pt}
\begin{align}
\|&x^{k+1}\!-x^\ast\|
 \overset{(i)}{\le} \|(\nabla^2\!f(x^k)+2\varepsilon_k I)^{-1}\nabla f(x^k)+x^k\!-x^\ast\|
   +\|d^k\!+(\nabla^2f(x^k)+2\varepsilon_k I)^{-1}\nabla\! f(x^k)\|\notag\\
&\le
   \|(\nabla^2\!f(x^k)+2\varepsilon_k I)^{-1}\|
   \Big(\|\nabla f(x^k)+\nabla^2f(x^k)(x^k-x^\ast)\|+2\varepsilon_k\|x^k-x^\ast\|
      +\zeta_k\|\nabla f(x^k)\|\Big)\notag\\ 
& \overset{(ii)}{\le}    \frac{2}{\mu}\Big(H_f\|x^k-x^\ast\|^{1+\nu}
   +2\gamma_k^\frac{1}{2}\|\nabla f(x^k)\|^\frac{1}{2}\|x^k-x^\ast\|
   +\|\nabla f(x^k)\|^\frac{3}{2}\Big),\label{ineq: upbd distance}
\end{align}
}

\noindent where (i) is by the triangle inequality, and (ii)  is by \eqref{ineq:desc-hd} and the definition of $\varepsilon_k,\zeta_k$ in Algorithm \ref{alg:NCG-hd}.
By reusing the step (ii) in the above inequality and  repeatedly using \eqref{eqn:loc-SC}, we have
\begin{equation*}\begin{aligned}
\|\nabla f(x^{k+1})\|
&\overset{(i)}{\le}
 \frac{2L_g}{\mu}\Big(H_f\|x^k\!-\!x^\ast\|^{1\!+\!\nu}
   \!+\!2\big(\gamma_\nu(\|\nabla f(x^k)\|)\|\nabla f(x^k)\|\big)^{\frac12}\|x^k\!-\!x^\ast\|
   \!+\!\|\nabla f(x^k)\|^\frac{3}{2}\Big)\\
   &{\le} \frac{2L_g}{\mu}\left(2H_f\mu^{\!-\!1}\|x^k\!-\!x^*\|^{\nu}\!+\!4H_f^{\frac{1}{1\!+\!\nu}}(2\mu^{\!-\!1})^{\frac{1}{1\!+\!\nu}}\|x^k\!-\!x^*\|^{\frac{\nu}{1\!+\!\nu}}\!+\!L_g^{\frac{1}{2}}\|x^k\!\!-\!\!x^*\|\right)\|\nabla f(x^k)\|\\
   &\overset{(ii)}{\le} \frac{2L_g}{\mu}\left(2H_f\mu^{\!-\!1}\!+\!4H^{\frac{1}{1\!+\!\nu}}(2\mu^{\!-\!1})^{\frac{1}{1\!+\!\nu}}\!+\!L_g^{\frac{1}{2}}\right)\|x^k\!-\!x^*\|^{\frac{\nu}{1\!+\!\nu}}\cdot\|\nabla f(x^k)\|\overset{(iii)}{\le}\frac{1}{3}\|\nabla f(x^k)\|,
\\
\end{aligned}\end{equation*}
where (i) uses $\gamma_k\leq \gamma_{\nu}(\|\nabla f(x^k)\|)$ for $k\in\mathbb K_{\epsilon,2}$, (ii) and (iii) are due to $x^k\in B_{\delta}(x^*)$.  This contradicts with $\|\nabla f(x^{k+1})\|\!>\!\frac{1}{2}\|\nabla f(x^k)\|$ for $k\!\in\!\mathbb K_{\epsilon,2}$. Thus, when $x^k\in B_{\delta}(x^*)$, we have $k\notin\mathbb K_{\epsilon,2} $, which implies that $\gamma_{k+1}=\gamma_k$.\hfill $\Box$

\subsubsection{Proof of Lemma \ref{lemma:capture}.}\emph{Proof.} Recall $\delta$ defined in \eqref{eqn:defn-delta-Thm4}, let $c\!=\!2\!(1\!+\!\frac{4L_g}{\mu}\!)$, and we define $S$ as 
\vspace{-2mm}
\begin{equation}\label{defn:S}
    S:=\Big\{x\;:\;\|x-x^*\|\leq\delta,\;f(x)-f(x^*)\leq\frac{\mu}{4}\big(\frac{\delta}{c}\big)^2\Big\}.
\end{equation}
Clearly, $S\subseteq B_{\delta}(x^\ast)$. By the previous discussion in Appendix \ref{proof: gamma_bound_universal}, we conclude that \eqref{eqn:loc-SC} holds for all $x\in S$, and Algorithm \ref{alg:NCG-hd} will always execute the Newton step with $\gamma_k=\gamma_{k+1}$ if $x_k\in S$.

We next prove that once $x_{k}\in S$, all future iterates will stay in $S$. For any $x^k\in S$, we have
\begin{equation*}\begin{aligned}
    \|\nabla f(x^k)\|\overset{(i)}{\geq}\frac{\zeta_k\varepsilon_k}{2}\|d^k\|\overset{(ii)}{\geq} &\|(\nabla^2f(x^k)+2\varepsilon_kI)d^k+\nabla f(x^k)\|\overset{(iii)}{\geq} \frac{\mu}{2}\|d^k\|-\|\nabla f(x^k)\|,
\end{aligned}\end{equation*}
where (i) is by $\zeta_k\leq1$ and the second inequality of Lemma~\ref{lem:ppt-cg}(i), (ii) is by the fourth inequality of Lemma~\ref{lem:ppt-cg}(i), and (iii) uses the triangle inequality and \eqref{eqn:loc-SC}. This implies that 
\(\|d^k\|\leq \frac{4}{\mu}\|\nabla f(x^k)\|\). And by \eqref{eqn:loc-SC}, we have $\frac{\mu}{4}\|x^k\!-\!x^*\|^2\!\leq\! f(x^k)\!-\!f(x^*)\!\leq\!\frac{\mu}{4}\big(\frac{\delta}{c}\big)^2$, which implies that $\|x^k\!-\!x^*\|\leq \frac{\delta}{c}$.
 Thus,
\begin{equation*}\begin{aligned}
    \|x^{k+1}-x^*\|=\|x^k+d^k-x^*\|
    \leq\|x^k-x^*\|+\|d^k\|
    \overset{(i)}{\leq} \Big(1+\frac{4L_g}{\mu}\Big)\|x^k-x^*\|<\delta,
\end{aligned}\end{equation*}
where (i) uses \eqref{eqn:loc-SC} and \(\|d^k\|\!\leq\! \frac{4}{\mu}\|\nabla\! f\!(x^k)\|\). Besides, since Algorithm \ref{alg:NCG-hd} is a descent method, one has
 $$f(x^{k+1})-f(x^*)\leq f(x^{k})-f(x^*)\leq \frac{\mu}{4}\Big(\frac{\delta}{c}\Big)^2 $$
Thus, $x^{k+1}\in S$. 
   Now suppose $x^{k_0}\in  S$ for some $k_0\geq 0$, by induction we have that $\{\!x^k\!\}_{k\geq k_0}\!\subseteq\! S$. \hfill $\Box$
\subsubsection{Proof of Theorem \ref{thm:loc-conv-2}.} \emph{Proof. } By Lemma \ref{lemma:gamma_bound_universal} and Lemma \ref{lemma:capture}, suppose $x^{k_0}\in S$, where $S$ is defined in \eqref{defn:S},  for some $k_0\ge 0$, we have that Algorithm \ref{alg:NCG-hd} will always execute the Newton step, $\{\!x^k\!\}_{k\geq k_0}\!\subseteq\! S$, and $\gamma_{k+1}=\gamma_k$ for $k\ge k_0$.  Therefore, $\gamma_k\!=\!\gamma_{k_0}$ for all $k\ge k_0$. Then, by \eqref{ineq: upbd distance} we have
\begin{equation*}\begin{aligned}
\|x^{k+1}\!-x^\ast\|
&\le\frac{2}{\mu}\Big(H_f\|x^k-x^\ast\|^{1+\nu}
   +2\gamma_k^\frac{1}{2}\|\nabla f(x^k)\|^\frac{1}{2}\|x^k-x^\ast\|
   +\|\nabla f(x^k)\|^\frac{3}{2}\Big)\\
   &\le\frac{2}{\mu}\Big(H_f\|x^k-x^\ast\|^{1+\nu}
   +(2\gamma_{k_0}^{\frac{1}{2}}+ L_g)L_g^{\frac{1}{2}}\|x^k-x^\ast\|^{\frac{3}{2}}\Big).\\
    \end{aligned}\end{equation*}
This shows $\|x^{k+1}\!-\!x^\ast\|=\mathcal{O}(\|x^{k}\!-\!x^\ast\|^{\min\{1+\nu,\frac{3}{2}\}})$ for $k\!\geq\! k_0$, i.e., a superlinear rate of order $\min\{1\!+\!\nu\!,\!\frac{3}{2}\}$.  

\hfill $\Box$
\section{Capped conjugate gradient method}\label{appendix:capped-CG}
In this work, we use the capped CG (\texttt{CappedCG}, Algorithm \ref{alg:capped-CG}) method \citep{royer2020newton} as a subroutine to solve the possibly \emph{indefinite} damped Newton systems. Consider the linear system $(H+2\sigma I)d = -g,$
where $g\neq0$ is a nonzero vector in $\bR^n$, $\sigma>0$, and $H\in\mathbb{S}^{n\times n}$ is a possibly indefinite symmetric matrix. \texttt{CappedCG} can efficiently return either (i) an approximate solution $d$ and d\_type=SOL with well controlled relative error or (ii) a negative curvature direction $d$ and d\_type=NC. The following lemma captures a few useful properties of \texttt{CappedCG}, which are taken from Lemma~3 of \citet{royer2020newton}.

\begin{algorithm}[h]
\caption{Capped conjugate gradient method $(d,\text{d\_type})=\texttt{CappedCG}(H, g, \zeta, \sigma)$}
\label{alg:capped-CG}
\footnotesize
\DontPrintSemicolon

\ShowLn\KwIn{symmetric matrix $H\in\bR^{n\times n}$, vector $g\neq0$, damping parameter $\sigma>0$, desired relative accuracy $\zeta\in(0,1)$.}
\ShowLn\KwOut{($d$, $\text{d\_type}$).\tcp*[r]{d\_type=SOL: solved, d\_type=NC: negative curvature}}   
\ShowLn\textbf{Initialize: }
\(U\leftarrow0,\;
\bar{H}\leftarrow H+2\sigma I,\;
\kappa\leftarrow\frac{U+2\sigma}{\sigma},\;
\hat{\zeta}\leftarrow\frac{\zeta}{3\kappa},\;
\tau\leftarrow\frac{\sqrt{\kappa}}{\sqrt{\kappa}+1},\;
T\leftarrow\frac{4\kappa^4}{(1-\sqrt{\tau})^2},
y^0\leftarrow 0,\ r^0\leftarrow g,\ p^0\leftarrow -g,\ j\leftarrow 0\).\;

\ShowLn \textbf{if} $(p^0)^{\rm T} \bar{H}p^0<\sigma\|p^0\|^2$ \textbf{then return} $(p^0, \rm NC)$.\; 

\ShowLn\While{\textnormal{TRUE}}{
  Compute $\alpha_j\leftarrow (r^j)^{\rm T} r^j/(p^j)^{\rm T}\bar{H}p^j$,\,\, 
  $y^{j+1}\leftarrow y^j+\alpha_jp^j$, and $r^{j+1}\leftarrow r^j+\alpha_j\bar{H}p^j$;\tcp*[r]{Begin Standard CG}
  Compute $\beta_{j+1}\leftarrow\|r^{j+1}\|^2/\|r^j\|^2$, and update  $p^{j+1}\leftarrow-r^{j+1}+\beta_{j+1}p^j$;\tcp*[r]{End Standard CG}
  Set $j\leftarrow j+1$, and update $U\leftarrow\max\left\{ U,\|Hp^0\|/\|p^0\|,\;\;\|Hp^j\|/\|p^j\|,\;\;\|Hy^j\|/\|y^j\|,\;\;\|Hr^j\|/\|r^j\| \right\}$; 
    \;Update $\kappa,\hat{\zeta},\tau,T$ by Line 3 accordingly.\; 
  \ShowLn\textbf{if} $(y^j)^{\rm T}\bar{H}y^j<\sigma\|y^j\|^2$ \textbf{then return} $(y^j, \rm NC)$.\; 
  \ShowLn\textbf{else if} $\|r^j\|\le\hat{\zeta}\|r^0\|$ \textbf{then return} $(y^j, \rm SOL)$.\; 
  \ShowLn\textbf{else if} $(p^j)^{\rm T}\bar{H}p^j<\sigma\|p^j\|^2$ \textbf{then return} $(p^j, \rm NC)$.\; 
  \ShowLn\ElseIf{$\|r^j\|>\sqrt{T}\tau^{j/2}\|r^0\|$}{
    Compute $\alpha_j, y^{j+1}$ as in the main loop above;\;
    Find $i\in\{0,\ldots,j-1\}$ such that
    \(
    (y^{j+1}-y^i)^{\rm T}\bar{H}(y^{j+1}-y^i)<\sigma\|y^{j+1}-y^i\|^2;
    \)\;
    \textbf{return $(y^{j+1}-y^i,\rm NC)$}. 
  }
}
\end{algorithm}
\begin{lemma}\label{lem:ppt-cg}
Let $(d,\text{d\_type})=\texttt{CappedCG}(H, g, \zeta, \sigma)$, then the following statements hold:
\begin{enumerate}[{\rm (i)}]
\item If d$\_$type=SOL, then  $d$ satisfies
\begin{equation*}\begin{aligned}
&\sigma\|d\|^2\le d^{\rm T}(H+2\sigma I)d,\qquad \|d\|\le 1.1\sigma^{-1}\| g\|,\\
&d^{\rm T} g=-d^{\rm T}(H+2 \sigma I)d,\qquad \|(H+2\sigma I)d+ g\|\le \zeta \sigma\|d\|/2.
\end{aligned}\end{equation*}
\item If d$\_$type=NC, then $d$ satisfies $d^{\rm T} g\le0$ and $d^{\rm T} H d/\|d\|^2<-\sigma$.
\end{enumerate}
\end{lemma}

\bibliographystyle{plainnat}
\bibliography{bib}

\end{document}